\documentclass[11pt]{amsart}

\usepackage[T1]{fontenc}
\usepackage{lmodern}
\usepackage{microtype}
\usepackage{amsmath,amssymb,amsthm,mathtools}
\usepackage[margin=1in]{geometry}
\usepackage{enumitem}
\usepackage[colorlinks=true,linkcolor=blue,citecolor=blue,urlcolor=blue]{hyperref}
\usepackage{aliascnt}

\newtheorem{theorem}{Theorem}[section]
\newaliascnt{proposition}{theorem}
\newtheorem{proposition}[proposition]{Proposition}
\aliascntresetthe{proposition}
\newaliascnt{lemma}{theorem}
\newtheorem{lemma}[lemma]{Lemma}
\aliascntresetthe{lemma}
\newaliascnt{corollary}{theorem}
\newtheorem{corollary}[corollary]{Corollary}
\aliascntresetthe{corollary}
\newaliascnt{hypothesis}{theorem}
\newtheorem{hypothesis}[hypothesis]{Hypothesis}
\aliascntresetthe{hypothesis}
\newaliascnt{question}{theorem}
\newtheorem{question}[question]{Question}
\aliascntresetthe{question}
\theoremstyle{definition}
\newaliascnt{definition}{theorem}

\aliascntresetthe{definition}
\newaliascnt{remark}{theorem}
\newtheorem{remark}[remark]{Remark}
\aliascntresetthe{remark}

\usepackage[nameinlink,capitalise]{cleveref}
\crefname{theorem}{Theorem}{Theorems}
\Crefname{theorem}{Theorem}{Theorems}
\crefname{proposition}{Proposition}{Propositions}
\Crefname{proposition}{Proposition}{Propositions}
\crefname{lemma}{Lemma}{Lemmas}
\Crefname{lemma}{Lemma}{Lemmas}
\crefname{corollary}{Corollary}{Corollaries}
\Crefname{corollary}{Corollary}{Corollaries}
\crefname{hypothesis}{Hypothesis}{Hypotheses}
\Crefname{hypothesis}{Hypothesis}{Hypotheses}
\crefname{question}{Question}{Questions}
\Crefname{question}{Question}{Questions}
\crefname{definition}{Definition}{Definitions}
\Crefname{definition}{Definition}{Definitions}
\crefname{remark}{Remark}{Remarks}
\Crefname{remark}{Remark}{Remarks}

\numberwithin{equation}{section}

\newcommand{\Pmin}{P^{-}}
\newcommand{\M}{\mathcal M}
\newcommand{\R}{\mathcal R}
\newcommand{\LL}{\mathcal L}
\newcommand{\UU}{\mathcal U}
\newcommand{\eps}{\varepsilon}
\newcommand{\one}{\mathbf 1}

\title[Average order for a shifted pairwise-coprime problem]{An Average-Order Theorem for a Shifted Pairwise-Coprime Extremal Problem}

\author[E. Li]{Eric Li}
\dedicatory{\normalfont\normalsize Trinity College, University of Cambridge}
\date{June 15, 2026}
\thanks{Email addresses: \href{mailto:el593@cantab.ac.uk}{el593@cantab.ac.uk}, \href{mailto:contact@ericli.com}{contact@ericli.com}.}
\thanks{The paper's forty named theorem, lemma, proposition, and corollary environments have been formally verified in Lean; see Section~\ref{sec:formalisation} and \cite{LeanFormalisation}.}


\hypersetup{
  pdftitle={An Average-Order Theorem for a Shifted Pairwise-Coprime Extremal Problem},
  pdfauthor={Eric Li},
  pdfsubject={Pairwise coprime sets and rough-number methods},
  pdfkeywords={pairwise coprime sets, Erdos problems, rough numbers, Buchstab function, sieve methods}
}

\begin{document}

\begin{abstract}
For $n\ge2$, let
\[
\M(n)=\sup_{\substack{A\subset[1,n)\\(a,b)=1\ (a\ne b)}}
       \sum_{a\in A}\frac1{n-a}.
\]
Erd\H{o}s asked whether $\M(n)\le \sum_{p<n}1/p+O(1)$ uniformly in $n$.
We prove the quantitative average order
\[
\sum_{n\le N}\M(n)=e^{-\gamma}N\log\log N+O(N).
\]
The proof combines the self-rough lower construction $\{n-d:\Pmin(n-d)>d\}$ with a bounded-cost dual certificate and Buchstab--de Bruijn estimates.  We further show
\[
\M(n)=(e^{-\gamma}+o(1))\log\log n
\]
for almost all $n$, with a quantitative exceptional-set bound, so Erd\H{o}s's inequality holds for almost all $n$.  The argument combines a long-interval two-dimensional beta sieve for two moving forbidden residue classes with an exact finite singular-series cancellation.  The remaining obstacles to a full variance estimate are a shifted-prime second moment for $d>\sqrt{2N}$ and finite-$u$ rough correlations in the intermediate tail, yielding a precise conditional criterion.  For every $\varepsilon>0$ we also prove the uniform pointwise bound
$\M(n)\le(2+\varepsilon)\log\log n+O_\varepsilon(1)$ and explain the linear-sieve barrier at the constant $2$.  The results of this paper have been formally verified in Lean.
\end{abstract}

\maketitle

\section{Introduction}

Let $[1,n)=\{1,2,\ldots,n-1\}$.  Erd\H{o}s asked whether every pairwise coprime set $A\subset[1,n)$ satisfies
\begin{equation}\label{eq:erdos-question}
\sum_{a\in A}\frac1{n-a}\leq \sum_{p<n}\frac1p+O(1),
\end{equation}
where $p$ runs over primes and the constant is absolute.  This is the corrected form of a question of Erd\H{o}s; historical and bibliographic context is collected later in Section~\ref{sec:historical-notes}.

Define
\begin{equation}\label{eq:M-def}
\M(n)=\sup_{\substack{A\subset [1,n)\\(a,b)=1\ (a\ne b)}}
\sum_{a\in A}\frac1{n-a}.
\end{equation}
The prime set
\[
A=\{p<n:p\in\mathbb P\}
\]
shows that shifted-prime reciprocal sums are an unavoidable special case.  Here and throughout, $\mathbb P$ denotes the set of primes.  Indeed, Erd\H{o}s Problem \#950 asks about
\[
f(n)=\sum_{p<n}\frac1{n-p},
\]
including whether $f(n)=o(\log\log n)$ for every $n$ \cite{Bloom950}.  The pointwise problem \eqref{eq:erdos-question} therefore contains hard short-interval prime phenomena.

The main unconditional result of this paper is an average theorem for the exact extremal quantity.

\begin{theorem}[Average order]\label{thm:main}
As $N\to\infty$,
\begin{equation}\label{eq:main-theorem}
\sum_{n\leq N}\M(n)=e^{-\gamma}N\log\log N+O(N).
\end{equation}
\end{theorem}

Since Mertens' theorem gives
\begin{equation}\label{eq:average-prime-harmonic}
\sum_{n\leq N}\sum_{p<n}\frac1p
=N\log\log N+O(N),
\end{equation}
Theorem \ref{thm:main} proves an averaged form of \eqref{eq:erdos-question}, with a positive proportional saving in the leading term.

\begin{corollary}[Average saving against the Erd\H{o}s benchmark]\label{cor:average-saving}
One has
\[
\sum_{n\leq N}\M(n)
= e^{-\gamma}\sum_{n\leq N}\sum_{p<n}\frac1p+O(N).
\]
In particular,
\[
\sum_{n\leq N}\M(n)
\leq \sum_{n\leq N}\sum_{p<n}\frac1p
-(1-e^{-\gamma})N\log\log N+O(N).
\]
\end{corollary}

The proof of Theorem \ref{thm:main} gives more structure than a first moment identity.  Let
\begin{equation}\label{eq:L-def-intro}
\LL(n)=\sum_{\substack{1\le d<n\\ \Pmin(n-d)>d}}\frac1d
\end{equation}
and
\begin{equation}\label{eq:U-def-intro}
\UU(n)=\sum_{16\le d<n}\frac1d\,
\one_{\{\Pmin(n-d)>d/(2\log d)\}}.
\end{equation}
The self-rough construction gives $\LL(n)\le \M(n)$ pointwise, and the bounded-cost dual certificate gives
\begin{equation}\label{eq:sandwich-intro}
\M(n)\le C+\UU(n)
\end{equation}
pointwise.  The theorem controls this sandwich only in mean:
\[
\sum_{n\le N}\LL(n)=e^{-\gamma}N\log\log N+O(N),
\qquad
\sum_{n\le N}\UU(n)=e^{-\gamma}N\log\log N+O(N).
\]
This caveat is important: Theorem \ref{thm:main} is not a proof of the pointwise Erd\H{o}s inequality.  It does, however, imply an almost-all reduction.

\begin{theorem}[Almost-all reduction to the self-rough sum]\label{thm:almost-all}
For almost all integers $n$,
\begin{equation}\label{eq:M-L-aa}
\M(n)=\LL(n)+o(\log\log n).
\end{equation}
More quantitatively, for every fixed $\eta>0$,
\begin{equation}\label{eq:M-L-exception-rate}
\#\{n\le N:\M(n)-\LL(n)>\eta\log\log n\}
\ll_\eta \frac{N}{\log\log N}.
\end{equation}
\end{theorem}

Thus the distribution of $\M(n)$, especially its upper edge, is essentially the distribution of the self-rough sum $\LL(n)$ for almost all $n$.  The second main result of the paper proves that this distribution is concentrated at the Buchstab value on the natural $\log\log n$ scale.

\begin{theorem}[Almost-all Erd\H{o}s inequality and normal order]\label{thm:main-aa}
For almost all integers $n$,
\begin{equation}\label{eq:main-aa-intro}
\M(n)=(e^{-\gamma}+o(1))\log\log n.
\end{equation}
More quantitatively, for every fixed $\eta>0$,
\begin{equation}\label{eq:main-aa-exception-intro}
\#\left\{n\le N:
\left|\M(n)-e^{-\gamma}\log\log n\right|>\eta\log\log n
\right\}
\ll_\eta \frac{N\log\log\log N}{\log\log N}.
\end{equation}
Consequently, there is a set $\mathcal N\subset\mathbb N$ of asymptotic density one and an absolute constant $C$ such that, for every sufficiently large $n\in\mathcal N$,
\[
\M(n)\le \sum_{p<n}\frac1p+C.
\]
\end{theorem}

The proof of Theorem \ref{thm:main-aa} is given in Section \ref{sec:distribution}.  It rests on a truncated second-moment estimate for $\LL(n)$ and a first-moment bound for the remaining tail.  The dependency chain is
\[
\begin{gathered}
\text{beta-sieve fundamental lemma}
\Longrightarrow \text{truncated second moment}\\
\Longrightarrow\; \LL\text{ normal order}
\Longrightarrow \M\text{ almost everywhere}.
\end{gathered}
\]
The load-bearing analytic input in that truncated $L^2$ estimate is a long-interval dimension-two beta-sieve fundamental lemma for finite systems of forbidden residue classes; the paper states the quoted sieve form and verifies its hypotheses before applying it.  More precisely, for suitable $Y=N^{\beta_N}$ with $\beta_N\to0$ and $\beta_N\log N\to\infty$, we prove an $O(N)$ variance bound for the truncated sum $\sum_{d\le Y}I_d(n)/d$ over $N<n\le2N$, which is more than sufficient for Chebyshev at the natural scale.  This is not a proof of the full untruncated $O(N)$ variance.  The key structural fact is parity: odd gaps give exact negative covariance, while every even-gap singular-series main term is positive and at least $2C_2>1$.  The needed cancellation is therefore even-versus-odd cancellation through a finite singular-series mean identity; absolute-value bounds lose the entire effect.  The full untruncated $O(N)$ variance would additionally require an $L^2$ treatment of the remaining high-level tail, where shifted-prime correlations enter.

The rest of the paper records tools for the exceptional-set and pointwise problem.  First, using the linear sieve at its edge, we prove the explicit pointwise estimate
\begin{equation}\label{eq:pointwise-2-intro}
\M(n)\le (2+\eps)\log\log n+O_\eps(1)
\end{equation}
for every $n$.  The constant $2$ is the honest constant delivered by this method.  Removing it by local counting alone would require prime-interval input of Hardy--Littlewood type; indeed the tempting estimate
\[
|A\cap[n-x,n)|\le \pi(x)+O(x/(\log x)^2)
\]
would imply
\[
\pi(X+Y)\le \pi(X)+\pi(Y)+O(Y/(\log Y)^2),
\]
a conjectural weakened form related to the second Hardy--Littlewood prime-interval problem; see the classical and later discussions in \cite{HardyLittlewood1923,HensleyRichards1974,Richards1974,MontgomeryVaughan1973}, and the contextual Erd\H{o}s Problem \#855 entry \cite{Bloom855}.

We then give a conditional formulation: a visibly conditional Hardy-Littlewood type window-packing hypothesis implies \eqref{eq:erdos-question}.  We also record a quantitative two-sided order-of-magnitude consequence on a positive-density set.  Finally, we include first-hit certificates, divisor-budget replacement criteria, exchange lemmas for forced elements, and CRT examples showing that direct divisor replacement can lose an unbounded amount.

\subsection{Formalisation}\label{sec:formalisation}
The forty named theorem, lemma, proposition, and corollary environments in this paper have been formalised in Lean~4 using mathlib.  The proof development is unconditional relative to Lean's standard classical foundations: it contains no admitted proof or project-defined axiom, and every conditional result is checked as an implication from its explicitly stated hypotheses.  For the quoted beta-sieve sequence corollary, the formal statement makes explicit the divisibility relation implicit in the paper's notation.  The source, reproducible build instructions, statement-by-statement theorem map, and kernel-axiom audit are available in the accompanying repository \cite{LeanFormalisation}.

\section{Sieve inputs and elementary sums}\label{sec:prelim}

For $m\geq2$, let $\Pmin(m)$ denote the least prime factor of $m$, and set $\Pmin(1)=\infty$.  Let
\begin{equation}\label{eq:Rxy-def}
\R(x,y)=\#\{m\leq x:\Pmin(m)>y\}.
\end{equation}
For real $z\ge2$ we write
\[
P(z)=\prod_{p\le z}p,
\]
where products over $p\le z$ are always over primes.  When a divisibility condition $q\mid h$ occurs, we use the usual convention that every integer $q$ divides $h=0$.
We shall use three standard sieve inputs.  The first is the Buchstab--de Bruijn theorem for integers free of small prime factors.  With
\[
u=\frac{\log x}{\log y},
\]
for every fixed $\delta_0>0$ it gives the uniform estimate
\begin{equation}\label{eq:buchstab-standard}
\R(x,y)=\frac{x\omega(u)}{\log y}+O_{\delta_0}\left(\frac{x}{(\log y)^2}+1\right)
\end{equation}
for $2\le y\le x$ with $x\ge y^{1+\delta_0}$.  This is the standard rough-number estimate recorded, in the notation $\Phi(x,y)$, in Tenenbaum \cite[Ch. III, \S6.2, Th. 3; Ch. III, \S6.4, Cor. 20]{Tenenbaum}.  We use \eqref{eq:buchstab-standard} only in ranges of the form $y\le x^\theta$ with fixed $\theta<1$, equivalently where the Buchstab parameter $u$ is bounded below by a number strictly larger than $1$; the cited theorem is uniform in such ranges.  The convention $\Pmin(1)=\infty$ changes the usual count by at most one, which is absorbed by the displayed error.  If one uses a version of the Buchstab--de Bruijn theorem with an additional term such as $O(y/\log y)$, that term is also absorbed in the error term in every application below: we only apply the estimate in fixed-power ranges $y\le x^\theta$ with $\theta<1$, so $y/\log y\ll_{\theta} x/(\log y)^2$ for large $y$.  Real values of $y$ cause no change beyond $O(1)$, since the set of primes $p\le y$ changes only when $y$ crosses a prime.  Here $\omega$ is Buchstab's function \cite{Buchstab}.  We also use the convergence estimate
\begin{equation}\label{eq:buchstab-decay-cited}
\omega(u)-e^{-\gamma}\ll \rho_{\rm D}(u)\qquad (u\ge1),
\end{equation}
where $\rho_{\rm D}$ is the Dickman--de Bruijn function; see Tenenbaum \cite[Ch. III, \S6.3, Th. 8]{Tenenbaum}.  In particular, since $\rho_{\rm D}(u)$ decays super-polynomially,
\begin{equation}\label{eq:buchstab-integral-cited}
\int_1^X\frac{\omega(u)-e^{-\gamma}}u\,du=O(1)
\end{equation}
uniformly in $X\ge1$.  The second input is the Rosser--Iwaniec linear sieve, in the form of Iwaniec and Kowalski \cite[Theorem 11.13]{IwaniecKowalski}.  The third is the one-dimensional Selberg upper-bound sieve; see Halberstam and Richert \cite[Chapter 2]{HalberstamRichert}.  The beta-sieve fundamental lemma with an exponential relative error in the sieve parameter is used in Section \ref{sec:distribution}; we cite Friedlander--Iwaniec \cite[Ch. 6, Fundamental Lemma]{FriedlanderIwaniec} or Iwaniec--Kowalski \cite[Theorem 6.9]{IwaniecKowalski}.  The uniformity in Lemma \ref{lem:rough}(i), namely $x\ge y^{U_0}$ with $y\to\infty$, is load-bearing in the proof of Theorem \ref{thm:main}.

\begin{lemma}[Buchstab--de Bruijn rough-number estimates]\label{lem:rough}
The following estimates hold.
\begin{enumerate}[label=\textup{(\roman*)}]
\item For every $\eps>0$ there are $U_0=U_0(\eps)$ and $y_0=y_0(\eps)$ such that, uniformly for $y\geq y_0$ and $x\geq y^{U_0}$,
\begin{equation}\label{eq:rough-asymp}
(e^{-\gamma}-\eps)\frac{x}{\log y}
\leq \R(x,y)
\leq
(e^{-\gamma}+\eps)\frac{x}{\log y}.
\end{equation}
Equivalently, this is the consequence of \eqref{eq:buchstab-standard} obtained from $u\ge U_0$ and $\omega(u)\to e^{-\gamma}$.
\item Uniformly for $x\geq1$ and $y\geq2$,
\begin{equation}\label{eq:rough-upper-global}
\R(x,y)\ll \frac{x}{\log y}+1.
\end{equation}
\end{enumerate}
\end{lemma}

\begin{lemma}[Linear sieve in short intervals]\label{lem:linear-sieve}
Fix $\alpha\in(1/3,1/2)$ and $\eps>0$.  Uniformly for all sufficiently large $D$, all intervals $I\subset[D,2D]$ with $|I|\asymp D$, and all integers $b$,
\begin{equation}\label{eq:linear-sieve-interval}
\#\{r\in I: p\nmid b-r\text{ for every prime }p\le D^\alpha\}
\le (2+\eps)\frac{|I|}{\log D}.
\end{equation}
The threshold for ``sufficiently large'' may depend on $\alpha$, $\eps$, and the implicit constants in $|I|\asymp D$.
\end{lemma}

\begin{proof}
We give the standard verification because the pointwise constant in Proposition \ref{prop:pointwise-two} rests on this input.  Let
\[
X=|I|,
\qquad z=D^\alpha,
\qquad Q=\frac{X}{(\log X)^B},
\]
where $B$ is fixed and large.  For every squarefree $q\le Q$ with prime factors at most $z$,
\[
\#\{r\in I:r\equiv b\pmod q\}=\frac{X}{q}+R_q,
\qquad |R_q|\le 1.
\]
Hence
\[
\sum_{q\le Q}|R_q|=O(Q)=o(X/\log D),
\]
uniformly in the residue class $b\pmod q$ and therefore uniformly in the integer $n$ used later.

Apply the upper-bound Rosser--Iwaniec linear sieve, for instance Iwaniec--Kowalski \cite[Theorem 11.13]{IwaniecKowalski}.  In the normalization used there, if a sequence of dimension $1$ has level $Q$, sieve level $z$, local density $1/p$ at each prime, and total remainder $o(XV(z))$, then its sifted part is at most
\[
X V(z)\{F(s)+o(1)\},\qquad s=\frac{\log Q}{\log z},
\]
where $F(s)=2e^\gamma/s$ for $2<s<3$.  We apply this theorem to the sequence of integers in $I$, excluding one residue class modulo each prime $p\le z$.  With
\[
s=\frac{\log Q}{\log z}=\frac1\alpha+o(1),
\]
we have $2<s<3$ for all large $D$, because $1/3<\alpha<1/2$.  In this range the upper linear-sieve function satisfies
\[
F(s)=\frac{2e^\gamma}{s}.
\]
Also Mertens' theorem gives
\[
V(z)=\prod_{p\le z}\left(1-\frac1p\right)
=(e^{-\gamma}+o(1))\frac1{\log z}
=(e^{-\gamma}+o(1))\frac1{\alpha\log D}.
\]
The linear-sieve bound therefore gives
\begin{align*}
X\,V(z)F(s)+o(X/\log D)
&= X\left(\frac{e^{-\gamma}+o(1)}{\alpha\log D}\right)
   \left(\frac{2e^\gamma}{1/\alpha+o(1)}\right)+o(X/\log D)\\
&=(2+o(1))\frac{X}{\log D}.
\end{align*}
The cancellation of $\alpha$ is important: within the admissible range $1/3<\alpha<1/2$, moving $\alpha$ does not improve the leading constant.  The convergence of the small-prime part in Proposition \ref{prop:pointwise-two} only requires $\alpha<1$.  The stricter condition $\alpha<1/2$ is imposed here to keep the sieve parameter $s>2$, while $\alpha>1/3$ keeps $s<3$ so that the displayed edge formula for $F$ is the relevant one.  Absorbing the $o(1)$ into $\eps$ proves \eqref{eq:linear-sieve-interval}.
\end{proof}

\begin{lemma}[Crude interval upper-bound sieve]\label{lem:crude-interval-sieve}
Uniformly for intervals $I$ of length $L\ge2$, integers $b$, and $2\le z\le L$,
\begin{equation}\label{eq:crude-interval-sieve}
\#\{m\in I:p\nmid b-m\text{ for every prime }p\le z\}\ll \frac{L}{\log z}.
\end{equation}
The implied constant is absolute.
\end{lemma}

\begin{proof}
This is the standard one-dimensional upper-bound sieve.  If $z\le L/(\log L)^B$, apply the Rosser--Iwaniec upper sieve with level $Q=L/(\log L)^B$ as in the proof of Lemma \ref{lem:linear-sieve}; here the upper-sieve function is bounded for $s=\log Q/\log z\ge1+o(1)$, and $\sum_{q\le Q}|R_q|=O(Q)$.  This gives $O(L/\log z)$.  If $L/(\log L)^B<z\le L$, then by monotonicity in the sifting range the sifted set is bounded by the same estimate with $z$ replaced by $L/(\log L)^B$, which is $O(L/\log L)=O(L/\log z)$.  Equivalently, one may quote the one-dimensional Selberg upper-bound sieve of Halberstam--Richert \cite[Chapter 2]{HalberstamRichert}.  The residue class is arbitrary throughout, so the estimate is uniform in $b$.
\end{proof}

We also use the standard sums
\begin{align}
\sum_{p\le x}\frac1p&=\log\log x+B_1+O(1/\log x),\label{eq:mertens-prime}\\
\sum_p\frac1{p\log p}&<\infty,\label{eq:prime-p-logp}\\
\sum_{3\le m\le X}\frac1{m\log m}&=\log\log X+O(1),\label{eq:int-loglog}\\
\sum_{16\le m\le X}\frac1{m\log(m/(2\log m))}&=\log\log X+O(1).\label{eq:z-m-sum}
\end{align}
The last identity follows from
\[
\frac1{\log(m/(2\log m))}=\frac1{\log m}+O\left(\frac{\log\log m}{(\log m)^2}\right).
\]

\begin{lemma}[Buchstab summation away from the transition range]\label{lem:buchstab-summation}
Let $z_d=d/(2\log d)$, and fix $0<\delta<1/4$.  As $N\to\infty$,
\begin{align}
\sum_{3\le d\le N^\delta}\frac{N-d}{d\log d}\,
\omega\!\left(\frac{\log(N-d)}{\log d}\right)
&=e^{-\gamma}N\log\log N+O_\delta(N),\label{eq:buchstab-sum-d}\\
\sum_{16\le d\le N^\delta}\frac{N-d}{d\log z_d}\,
\omega\!\left(\frac{\log(N-d)}{\log z_d}\right)
&=e^{-\gamma}N\log\log N+O_\delta(N).\label{eq:buchstab-sum-z}
\end{align}
Moreover
\begin{align}
\sum_{N^\delta<d\le N/2}\frac{N}{d\log d}&=O_\delta(N),\label{eq:buchstab-transition-d}\\
\sum_{N^\delta<d\le N/2}\frac{N}{d\log z_d}&=O_\delta(N).\label{eq:buchstab-transition-z}
\end{align}
\end{lemma}

\begin{proof}
We first prove \eqref{eq:buchstab-sum-d}.  Since $d\le N^\delta$ and $\delta<1$, the Buchstab parameter
$\log(N-d)/\log d$ is bounded below by a constant greater than $1$.  Put $U=\log N$.  We shall use the following elementary regularity consequence of Buchstab's differential equation.  For fixed $\eta>0$, the function $u\omega(u)$ has bounded derivative on each interval of continuity in $[1+\eta,\infty)$ and total variation $O_\eta(U)$ on $[1+\eta,U]$; indeed $(u\omega(u))'=\omega(u-1)$ for $u>2$, while the remaining compact interval is harmless.  Hence $t\mapsto \omega(U/t)/t=(U/t)\omega(U/t)/U$ has total variation $O_\delta(1)$ on $[O(1),\delta U]$, since $U/t\ge1/\delta+O(1)$ there.
By partial summation, using $\sum_{m\le x}1/m=\log x+\gamma+O(1/x)$, one has
\[
\sum_{3\le d\le N^\delta}\frac1{d\log d}
\omega\!\left(\frac{\log N}{\log d}\right)
=
\int_{O(1)}^{\delta U}\frac1t\omega\!\left(\frac Ut\right)dt+O_\delta(1).
\]
With $u=U/t$, this integral is
\[
\int_{1/\delta}^{U/O(1)}\frac{\omega(u)}u\,du+O_\delta(1)
=e^{-\gamma}\log U+O_\delta(1)
=e^{-\gamma}\log\log N+O_\delta(1),
\]
where the penultimate equality uses \eqref{eq:buchstab-integral-cited}.  It remains to replace $\log N$ by $\log(N-d)$ and $N$ by $N-d$.  Since
\[
\log(N-d)=\log N+O(N^{\delta-1})
\]
and the arguments of $\omega$ remain in a compact subset of $(1,\infty)$ together with a ray to infinity, Buchstab's differential equation and boundedness of $\omega$ give
\[
\left|\omega\!\left(\frac{\log(N-d)}{\log d}\right)-
\omega\!\left(\frac{\log N}{\log d}\right)\right|
\ll_\delta \frac{N^{\delta-1}}{\log d}.
\]
After multiplying by $N/(d\log d)$ and summing over $d\le N^\delta$, this contributes $O_\delta(N^\delta)=O_\delta(N)$.  Replacing the factor $N$ by $N-d$ contributes
\[
\sum_{d\le N^\delta}\frac{1}{\log d}=O_\delta(N^\delta)=O_\delta(N).
\]
This proves \eqref{eq:buchstab-sum-d}.

We now prove \eqref{eq:buchstab-sum-z}; this is the only point where the change from $d$ to $z_d=d/(2\log d)$ requires care.  Write
\[
t=\log d,\qquad v=\log z_d=t-\log(2t).
\]
For $d\ge16$, $v\asymp t$ and
\[
\frac{dv}{dt}=1-\frac1t,\qquad
 dt=\left(1+O\!\left(\frac1v\right)\right)dv.
\]
A partial-summation argument exactly as above gives
\[
\sum_{16\le d\le N^\delta}\frac1{d\log z_d}
\omega\!\left(\frac{\log N}{\log z_d}\right)
=
\int_{v_0}^{\delta U+O(\log U)}
\frac{\omega(U/v)}{v}\left(1+O\!\left(\frac1v\right)\right)dv
+O_\delta(1),
\]
where $v_0>0$ is absolute.  The error integral is bounded by $\int_{v_0}^\infty v^{-2}\,dv$, since $\omega$ is bounded.  In the main integral, the change of variables $u=U/v$ gives
\[
\int_{1/\delta+O((\log U)/U)}^{U/v_0}\frac{\omega(u)}u\,du+O_\delta(1)
=e^{-\gamma}\log U+O_\delta(1).
\]
The perturbation from $\log N$ to $\log(N-d)$ is bounded as in the first part, now with $\log z_d\asymp\log d$, and contributes $O_\delta(N)$.  Replacing $N$ by $N-d$ contributes $O_\delta(N)$ as well.  This proves \eqref{eq:buchstab-sum-z}.

Finally, \eqref{eq:buchstab-transition-d} and \eqref{eq:buchstab-transition-z} follow from integral comparison:
\[
N\sum_{N^\delta<d\le N/2}\frac1{d\log d}
=N\log(1/\delta)+O(N),
\]
and the same estimate with $\log z_d$ in place of $\log d$, since $\log z_d=\log d+O(\log\log d)$ in the transition range.
\end{proof}

\section{The mean sandwich: self-rough lower bound and bounded-cost upper bound}\label{sec:main-proof}

For $n\geq2$ define
\[
A_R(n)=\{n-d:1\le d<n,\ \Pmin(n-d)>d\}.
\]

\begin{lemma}[Self-rough construction]\label{lem:self-rough}
The set $A_R(n)$ is pairwise coprime.  Consequently,
\begin{equation}\label{eq:L-le-M}
\LL(n)\le \M(n).
\end{equation}
\end{lemma}

\begin{proof}
Let $d_1<d_2$, and suppose that a prime $q$ divides both $n-d_1$ and $n-d_2$.  Then
\[
q\mid (n-d_1)-(n-d_2)=d_2-d_1,
\]
so $q<d_2$.  But $q\mid n-d_2$, contradicting $\Pmin(n-d_2)>d_2$.
\end{proof}

\begin{proposition}[Average lower bound]\label{prop:lower}
As $N\to\infty$,
\begin{equation}\label{eq:L-average-lower}
\sum_{n\le N}\LL(n)=e^{-\gamma}N\log\log N+O(N).
\end{equation}
In particular,
\[
\sum_{n\le N}\M(n)\ge e^{-\gamma}N\log\log N+O(N).
\]
\end{proposition}

\begin{proof}
Changing variables $m=n-d$ gives
\begin{equation}\label{eq:L-average-exact}
\sum_{n\le N}\LL(n)=\sum_{d<N}\frac1d\R(N-d,d).
\end{equation}
Fix once and for all a small number $0<\delta<1/4$.  For $d>N/2$ the contribution is $O(N)$, since $\R(N-d,d)\le N-d+1$.  For $N^\delta<d\le N/2$, the global upper-bound sieve estimate \eqref{eq:rough-upper-global} gives
\[
\sum_{N^\delta<d\le N/2}\frac1d\R(N-d,d)
\ll N\sum_{N^\delta<d\le N/2}\frac1{d\log d}=O_\delta(N).
\]
Choose $d_0=d_0(\delta)\ge3$ so that the quoted Buchstab--de Bruijn estimate is valid for all $d\ge d_0$ in the range below.  The finitely many terms $1\le d<d_0$ contribute $O_\delta(N)$ trivially.  For $d_0\le d\le N^\delta$ the Buchstab parameter is bounded below by a constant greater than $1$, so \eqref{eq:buchstab-standard} applies uniformly and gives
\[
\R(N-d,d)=
\frac{(N-d)\omega(\log(N-d)/\log d)}{\log d}
+O_\delta\left(\frac{N}{(\log d)^2}+1\right).
\]
The accumulated error, after multiplication by $1/d$ and summation over $d\le N^\delta$, is $O_\delta(N)$.  Lemma \ref{lem:buchstab-summation} gives the main term
$e^{-\gamma}N\log\log N+O_\delta(N)$.  Since $\delta$ is fixed, this proves \eqref{eq:L-average-lower}.
\end{proof}

The upper bound is based on a dual certificate of bounded total cost.

\begin{lemma}[Bounded-cost dual certificate]\label{lem:bounded-dual}
There is an absolute constant $C_0$ such that, for every $n\geq3$,
\begin{equation}\label{eq:bounded-dual}
\M(n)\le C_0+\UU(n).
\end{equation}
\end{lemma}

\begin{proof}
Let $A\subset[1,n)$ be pairwise coprime and write $r=n-a$.  The terms with $r<16$ and the possible element $a=1$ contribute $O(1)$.  For $r\ge16$, if $a=n-r>1$ has a prime divisor $p\le r/(2\log r)$, then
\[
p\log p\le \frac{r}{2\log r}\log r=\frac r2,
\]
and hence
\[
\frac1r\le \frac1{p\log p}\le \sum_{q\mid a}\frac1{q\log q}.
\]
If no such prime divisor exists, then $\Pmin(a)>r/(2\log r)$ and we keep the residual term $1/r$.  Thus
\[
\frac1r\le \sum_{q\mid a}\frac1{q\log q}
+\frac1r\one_{\{\Pmin(a)>r/(2\log r)\}}.
\]
Since the elements of $A$ are pairwise coprime, their prime divisor sets are disjoint.  Therefore
\[
\sum_{a\in A}\sum_{q\mid a}\frac1{q\log q}\le\sum_q\frac1{q\log q}<\infty.
\]
Summing over $a\in A$ proves \eqref{eq:bounded-dual}.
\end{proof}

\begin{proposition}[Average upper bound]\label{prop:upper}
As $N\to\infty$,
\begin{equation}\label{eq:U-average-upper}
\sum_{n\le N}\UU(n)=e^{-\gamma}N\log\log N+O(N).
\end{equation}
Consequently,
\[
\sum_{n\le N}\M(n)\le e^{-\gamma}N\log\log N+O(N).
\]
\end{proposition}

\begin{proof}
Put $z_d=d/(2\log d)$ for $d\ge16$.  Then
\begin{equation}\label{eq:U-average-exact}
\sum_{n\le N}\UU(n)=\sum_{16\le d<N}\frac1d\R(N-d,z_d).
\end{equation}
Fix $0<\delta<1/4$.  For $d>N/2$ the contribution is $O(N)$ by the trivial bound $\R(N-d,z_d)\le N-d+1$.  For $N^\delta<d\le N/2$, \eqref{eq:rough-upper-global} gives
\[
\sum_{N^\delta<d\le N/2}\frac1d\R(N-d,z_d)
\ll N\sum_{N^\delta<d\le N/2}\frac1{d\log z_d}=O_\delta(N).
\]
Choose $d_0=d_0(\delta)\ge16$ so large that $z_d=d/(2\log d)$ is at least the threshold required in the quoted Buchstab--de Bruijn estimate for every $d\ge d_0$ in the range below.  The finitely many terms $16\le d<d_0$ contribute $O_\delta(N)$.  For $d_0\le d\le N^\delta$, the Buchstab parameter
$\log(N-d)/\log z_d$ is bounded below by a constant greater than $1$, so \eqref{eq:buchstab-standard} gives
\[
\R(N-d,z_d)=
\frac{(N-d)\omega(\log(N-d)/\log z_d)}{\log z_d}
+O_\delta\left(\frac{N}{(\log z_d)^2}+1\right).
\]
The error contributes $O_\delta(N)$, and Lemma \ref{lem:buchstab-summation} gives \eqref{eq:U-average-upper}.  Lemma \ref{lem:bounded-dual} then gives the asserted bound for $\M$.
\end{proof}

\begin{proof}[Proof of Theorem \ref{thm:main}]
Combine Propositions \ref{prop:lower} and \ref{prop:upper} with the pointwise inequalities $\LL(n)\le\M(n)\le C_0+\UU(n)$.
\end{proof}

\begin{proof}[Proof of Corollary \ref{cor:average-saving}]
By Mertens' theorem,
\[
\sum_{n\le N}\sum_{p<n}\frac1p
=\sum_{p<N}\frac{N-p}{p}
=N\sum_{p<N}\frac1p-\pi(N)+O(1)
=N\log\log N+O(N).
\]
The corollary follows from Theorem \ref{thm:main}.
\end{proof}

\section{Almost-all reduction and the variance problem}\label{sec:distribution}

We now prove Theorem \ref{thm:almost-all}.  The proof uses only the mean sandwich from Section \ref{sec:main-proof}.

\begin{proof}[Proof of Theorem \ref{thm:almost-all}]
By Propositions \ref{prop:lower} and \ref{prop:upper},
\[
\sum_{n\le N}\LL(n)=e^{-\gamma}N\log\log N+O(N),
\qquad
\sum_{n\le N}\UU(n)=e^{-\gamma}N\log\log N+O(N).
\]
Let $C_1=\sum_{1\le d<16}1/d$.  Since $\Pmin(n-d)>d$ implies $\Pmin(n-d)>d/(2\log d)$ for $d\ge16$, we have
\[
0\le C_1+\UU(n)-\LL(n)
\]
for every $n$.  Moreover
\begin{equation}\label{eq:mean-gap-UL}
\sum_{n\le N}(C_1+\UU(n)-\LL(n))=O(N),
\end{equation}
because the additional $C_1N$ term has the same size.  For every fixed $\eta>0$, Markov's inequality gives
\[
\#\{n\le N:C_1+\UU(n)-\LL(n)>\eta\log\log N\}
\ll_\eta \frac{N}{\log\log N}.
\]
By Lemma \ref{lem:bounded-dual},
\[
0\le \M(n)-\LL(n)\le C_0+\UU(n)-\LL(n)
\le C_0+C_1+\UU(n)-\LL(n),
\]
so the same estimate holds with $\M(n)-\LL(n)$ in place of $C_1+\UU(n)-\LL(n)$.  This proves
\[
\M(n)=\LL(n)+o(\log\log N)
\]
for almost all $n\le N$.  To replace $\log\log N$ by $\log\log n$, discard the $O(\sqrt N)=o(N)$ integers $n\le\sqrt N$.  On the remaining range, for large $N$,
\[
\log\log n\ge \tfrac12\log\log N.
\]
Hence an exceptional set on which $\M(n)-\LL(n)>\eta\log\log n$ would also be exceptional for the preceding estimate with threshold $(\eta/2)\log\log N$.  This gives \eqref{eq:M-L-aa}.
\end{proof}

Theorem \ref{thm:almost-all} isolates the distributional problem to be solved by a second moment.  We shall use the elementary pointwise estimate
\begin{equation}\label{eq:L-pointwise-O}
\LL(n)\ll\log\log n,
\end{equation}
which follows from Lemma \ref{lem:crude-interval-sieve} by summing dyadic blocks: in $D<d\le2D$ the number of $d$ with $\Pmin(n-d)>d$ is $O(D/\log D)$.

The natural full variance conjecture is the dyadic estimate
\begin{equation}\label{eq:L-second-moment-O-N-heuristic}
\sum_{N<n\le2N}\left|\LL(n)-e^{-\gamma}\log\log N\right|^2=O(N).
\end{equation}
For Theorem \ref{thm:main-aa}, however, it is enough to prove the weaker normalized estimate
\begin{equation}\label{eq:L-second-moment-dyadic}
\sum_{N<n\le2N}\left|\LL(n)-e^{-\gamma}\log\log N\right|^2
=o\bigl(N(\log\log N)^2\bigr).
\end{equation}
The word ``enough'' uses the pointwise bound \eqref{eq:L-pointwise-O}, which controls the contribution of exceptional sets to the second moment.  Passing between dyadic intervals and $n\le N$ is standard: discard $n\le\sqrt N$, which is $o(N)$, and use $\log\log n\ge \tfrac12\log\log N$ on the remaining range.

The first moment plus Markov's inequality controls only the mean gap between $\M$ and $\LL$.  It gives no almost-all upper bound for $\LL$ itself.  A variance estimate, or an equally strong substitute, is therefore a genuinely new input.  Theorem \ref{thm:truncated-second-moment} below proves the required normalized estimate by truncating at $Y=N^{\beta_N}$.  The remaining tail is then negligible for almost all $n$ by its first moment.

We now record the exact singular-series structure of the off-diagonal.  This is the organizing fact for the second-moment proof of Theorem \ref{thm:main-aa}.

\begin{remark}[Parity and the twin-prime singular series: motivation only]\label{rem:parity-singular-series}
For $h\ge1$ define
\begin{equation}\label{eq:singular-series-h-new}
\mathfrak S(h)=
\begin{cases}
2C_2\displaystyle\prod_{\substack{p\mid h\\p>2}}\dfrac{p-1}{p-2},& h\text{ even},\\[1.2ex]
0,&h\text{ odd},
\end{cases}
\qquad
C_2=\prod_{p>2}\left(1-\frac1{(p-1)^2}\right).
\end{equation}
Then $\mathfrak S(h)=0$ for odd $h$, while
\begin{equation}\label{eq:S-positive-even}
\mathfrak S(h)\ge 2C_2>1.3203\ldots>1
\end{equation}
for every even $h$.  Moreover
\begin{equation}\label{eq:singular-series-average-new}
\sum_{h\le H}\mathfrak S(h)=H+O((\log H)^2).
\end{equation}
\end{remark}

\begin{proof}
The formula \eqref{eq:singular-series-h-new} is the usual twin-prime singular series for the pair of linear forms $m,m+h$.  If $h$ is odd, the two forms have opposite parity, so they cannot both avoid the residue class $0\pmod 2$; this is the local obstruction at $2$ and gives $\mathfrak S(h)=0$.  If $h$ is even, every factor $(p-1)/(p-2)$ in \eqref{eq:singular-series-h-new} is at least $1$, giving \eqref{eq:S-positive-even}; the numerical value follows from the classical evaluation $C_2=0.66016\ldots$.

The mean-value estimate \eqref{eq:singular-series-average-new} is the classical Gallagher mean-value theorem for singular series \cite{Gallagher1976}.  In this special two-form case it also follows by expanding the product into a divisor sum and summing the resulting local factors; the complete-period average of the finite local factor is exactly $1$ prime by prime, and Gallagher's argument gives the displayed $O((\log H)^2)$ remainder.
\end{proof}

\begin{remark}[What the covariance calculation really says]\label{rem:covariance}
Write
\[
I_d(n)=\one_{\{\Pmin(n-d)>d\}},
\qquad
\LL(n)=\sum_{d<n}\frac{I_d(n)}d.
\]
The diagonal part of the second moment is harmless:
\[
\sum_d \frac{\mathbb E I_d}{d^2}=O(1).
\]
For the off-diagonal, let $d_1<d_2$ and $h=d_2-d_1$.  The average of $I_{d_1}(n)I_{d_2}(n)$ counts integers $m=n-d_2$ for which $m$ is $d_2$-rough and $m+h$ is $d_1$-rough.  The local two-point sieve density has the formal finite product
\begin{equation}\label{eq:two-point-local-density}
\prod_{\substack{p\le d_1\\p\nmid h}}\left(1-\frac2p\right)
\prod_{\substack{p\le d_1\\p\mid h}}\left(1-\frac1p\right)
\prod_{d_1<p\le d_2}\left(1-\frac1p\right).
\end{equation}
After division by the product of the one-point densities, the finite correlation factor is
\begin{equation}\label{eq:two-point-correlation-factor}
C_{d_1}(h)=
\prod_{\substack{p\le d_1\\p\nmid h}}\left(1-\frac1{(p-1)^2}\right)
\prod_{\substack{p\le d_1\\p\mid h}}\frac{p}{p-1}.
\end{equation}
As $d_1\to\infty$, this tends to $\mathfrak S(h)$ for each fixed $h$.  Thus the covariance is expected to have the shape
\begin{equation}\label{eq:covariance-shape}
\operatorname{Cov}(I_{d_1},I_{d_2})
\approx
\frac{e^{-2\gamma}}{\log d_1\log d_2}\bigl(\mathfrak S(h)-1\bigr).
\end{equation}

The important correction is that the cancellation is not sign-mixing inside a single parity class.  For odd $h$, and $d_1,d_2\ge2$, the product $I_{d_1}(n)I_{d_2}(n)$ is identically zero: $n-d_1$ and $n-d_2$ differ by an odd number, so exactly one of them is even, while a $d$-rough integer with $d\ge2$ must be odd.  Hence the odd-$h$ covariance is negative, of size
\[
-\frac{e^{-2\gamma}}{\log d_1\log d_2}(1+o(1)).
\]
For even $h$, Remark \ref{rem:parity-singular-series} gives $\mathfrak S(h)>1$, so the main-term covariance is positive.  The even-$h$ covariances alone have total size comparable to $(\log\log N)^2$.  The collapse predicted in \eqref{eq:L-second-moment-O-N-heuristic} is caused by the exact even-versus-odd balance expressed by
\[
\sum_{h\le H}(\mathfrak S(h)-1)=O((\log H)^2),
\]
not by absolute-value estimates.

Consequently the second-moment proof should have the following skeleton.
\begin{enumerate}[label=\textup{(\roman*)}]
\item Prove a uniform two-dimensional fundamental lemma for the pair of forms $m,m+h$ with split sieving levels $d_2>d_1$.  The main term should be the finite product \eqref{eq:two-point-local-density}.  The proof below then sums the finite factor $C_{d_1}(h)$ directly; no limiting passage to the full singular series is used.
\item Use the parity vanishing for odd $h$ exactly; this part is free and should not be replaced by an absolute-value bound.
\item Sum the remaining main terms by the exact finite-product mean identity of Lemma \ref{lem:finite-singular-series-average}.  This is the step that cancels the positive even-$h$ mass against the negative odd-$h$ ballast.
\end{enumerate}

At the limiting heuristic level, partial summation in $h$, using \eqref{eq:singular-series-average-new}, gives the model estimate
\[
\sum_{d_2>d_1}\frac{\mathfrak S(d_2-d_1)-1}{d_2\log d_2}
\ll \frac{\log d_1}{d_1},
\]
and therefore
\[
\sum_{d_1<d_2}\frac{\operatorname{Cov}(I_{d_1},I_{d_2})}{d_1d_2}=O(1).
\]
This predicts bounded fluctuations of $\LL(n)-e^{-\gamma}\log\log n$ on dyadic intervals.  The resulting limiting distribution, if it exists, should be a fixed $O(1)$-scale law; it should not be sought at scale $\sqrt{\log\log n}$ or $\log\log n$.
\end{remark}

The preceding discussion is only motivation.  The proof below does not pass to the limiting singular series $\mathfrak S(h)$.  Instead it works with the finite local factor $C_y(h)$ and the exact finite mean identity $\sum_{q\mid P(y)}a_y(q)/q=1$, which avoids any non-uniform truncation in $h$.

The preceding discussion identifies the precise technical statement needed for the almost-all theorem.  We now prove it.  The almost-all theorem depends on the following long-interval two-dimensional fundamental-lemma input; this is the load-bearing analytic step of the section.  The only external input beyond the one-dimensional sieve estimates already used is the beta-sieve fundamental lemma itself.  The formal imported form is stated in Appendix \ref{app:beta-sieve} as Theorem \ref{thm:imported-beta-sieve}; the next lemma is the direct sequence-level consequence in the notation of this paper.  We then verify its hypotheses for the moving residue classes arising from the two forms $m$ and $m+h$.  Thus the result below is unconditional once the standard beta-sieve fundamental lemma is imported in the quoted form; no prime-distribution hypothesis is used in the proof of the almost-all theorem.

\begin{lemma}[Beta-sieve fundamental lemma, sequence form]\label{lem:quoted-beta-sieve}
Let $z\ge2$, $Q\ge z$, and $s=\log Q/\log z$.  Let $g$ be a multiplicative function supported on squarefree divisors of $P(z)=\prod_{p\le z}p$, with $0\le g(p)<1$.  Suppose that for some fixed constants $\kappa\ge1$ and $C_\kappa$ one has
\begin{equation}\label{eq:quoted-dimension}
\prod_{w\le p<z}(1-g(p))^{-1}
\le C_\kappa\left(\frac{\log z}{\log w}\right)^\kappa
\qquad(2\le w\le z).
\end{equation}
There is a threshold $s_0=s_0(\kappa,C_\kappa)$ in the beta-sieve fundamental lemma.  Assume $s\ge s_0$; in the applications below $s\to\infty$.  Then there are upper and lower beta-sieve weights $\lambda_q^\pm$, supported on squarefree $q\mid P(z)$ with $q\le Q$ and satisfying $\lambda_1^\pm=1$ and $|\lambda_q^\pm|\le1$, such that for every finite sequence \(\mathcal A\) equipped with sets \(\mathcal A_q\) satisfying
\[
|\mathcal A_q|=Xg(q)+r_q\qquad(q\mid P(z),\ q\le Q),
\]
and for the sifted set \(S(\mathcal A,z)\) whose elements avoid all primes \(p\le z\), one has
\[
\sum_{q\mid P(z)}\lambda_q^- |\mathcal A_q|
\le S(\mathcal A,z)
\le
\sum_{q\mid P(z)}\lambda_q^+ |\mathcal A_q|,
\]
and
\begin{equation}\label{eq:quoted-beta-main}
\sum_{q\mid P(z)}\lambda_q^\pm g(q)
=\prod_{p\le z}(1-g(p))\left(1+O_{\kappa,C_\kappa}(e^{-c_\kappa s})\right).
\end{equation}
Consequently
\begin{equation}\label{eq:quoted-beta-consequence}
S(\mathcal A,z)
=X\prod_{p\le z}(1-g(p))
\left(1+O_{\kappa,C_\kappa}(e^{-c_\kappa s})\right)
+O\left(\sum_{q\le Q}|r_q|\right).
\end{equation}
The conclusion is obtained by sandwiching $S(\mathcal A,z)$ between the lower and upper weighted sums; the relative error in the main term is therefore relative to the local product $\prod_{p\le z}(1-g(p))$, while all non-periodic distribution error enters only through $\sum_{q\le Q}|r_q|$.  
This is the sequence form obtained from the imported beta-sieve theorem, Theorem \ref{thm:imported-beta-sieve}.  In the notation of Friedlander--Iwaniec and Iwaniec--Kowalski, our $Q$ is the distribution level usually denoted $D$, our $z$ is the sifting level, and $s=\log Q/\log z$ is the sieve parameter.  The support condition $q\mid P(z)$, $q\le Q$, and the inequalities $|\lambda_q^\pm|\le1$ are the support and size conditions of the beta-sieve weights.  The theorem is uniform over all multiplicative local densities satisfying the displayed dimension condition.  Thus, once the dimension constant is fixed, the constants in \eqref{eq:quoted-beta-main} are independent of the particular residue classes that give rise to $g$.  In the applications below $\kappa=2$, $Q=N^{1/2}$, $z\le N^{\beta_N}$, and hence $s\ge(2\beta_N)^{-1}\to\infty$.  For the concrete choice $\beta_N=(\log\log N)^{-1/2}$ used later,
\[
E_2(s)(\log\log N)^2
\ll e^{-c\sqrt{\log\log N}}(\log\log N)^2=o(1),
\]
which is the only quantitative strength of the fundamental lemma needed in the second-moment summation.
\end{lemma}

\begin{lemma}[Long-interval beta sieve with two forbidden classes]\label{lem:long-interval-beta-sieve}
Let $I$ be an interval of integers of length $N$.  Let $2\le z\le N^{\beta_N}$, where $\beta_N\to0$ and $\beta_N\log N\to\infty$.  For every prime $p\le z$ let $\Omega_p\subset \mathbb Z/p\mathbb Z$ be a set of forbidden residue classes, and put $\nu(p)=|\Omega_p|$.  Assume
\begin{enumerate}[label=\textup{(\alph*)}]
\item $0\le \nu(p)<p$ and $\nu(p)\le2$ for every $p\le z$;
\item for each squarefree $q$ supported on primes $p\le z$, if $\Omega_q$ is the set of residue classes modulo $q$ obtained by the Chinese remainder theorem and $\nu(q)=|\Omega_q|$, then
\begin{equation}\label{eq:generic-sieve-remainder}
\#\{m\in I:m\bmod q\in \Omega_q\}
=N\frac{\nu(q)}q+r_q,
\qquad |r_q|\le 2^{\omega(q)} .
\end{equation}
\end{enumerate}
Let
\begin{equation}\label{eq:generic-sieve-product}
V_\Omega(z)=\prod_{p\le z}\left(1-\frac{\nu(p)}p\right).
\end{equation}
Then, uniformly in the sets $\Omega_p$,
\begin{equation}\label{eq:generic-beta-sieve}
\#\{m\in I:m\bmod p\notin\Omega_p\text{ for every }p\le z\}
=N V_\Omega(z)\left(1+O(e^{-c/\beta_N})\right)+O(N^{1/2+o(1)}),
\end{equation}
where $c>0$ is absolute.  The $o(1)$ in the final error is absolute and uniform.
\end{lemma}

\begin{proof}
Let $Q=N^{1/2}$ and $P(z)=\prod_{p\le z}p$.  If $z$ is bounded, say $z\le z_0$, the Chinese remainder theorem gives $S=NV_\Omega(z)+O_{z_0}(1)$ uniformly for all such $z$, which is absorbed by the asserted $N^{1/2+o(1)}$ term for large $N$.  We may therefore assume $z\to\infty$.  By \eqref{eq:generic-sieve-remainder},
\begin{equation}\label{eq:generic-remainder-sum}
\sum_{\substack{q\le Q\\ q\text{ squarefree}}}|r_q|
\le \sum_{q\le Q}2^{\omega(q)}=N^{1/2+o(1)}.
\end{equation}
The dimension condition needed for the beta sieve is uniform.  Since $\nu(p)\le2$ and $\nu(p)<p$,
\begin{equation}\label{eq:generic-dimension-bound}
\prod_{w\le p<z}\left(1-\frac{\nu(p)}p\right)^{-1}
\le C\left(\frac{\log z}{\log w}\right)^2
\end{equation}
for $2\le w\le z$, with an absolute $C$.  Indeed, the prime $2$ contributes at most an absolute factor, and for $p>2$,
\[
\log\left(1-\frac{\nu(p)}p\right)^{-1}
\le \frac2p+O\left(\frac1{p^2}\right),
\]
so Mertens' estimate for $\sum_{w\le p<z}1/p$ gives \eqref{eq:generic-dimension-bound}.  Let $g(q)=\nu(q)/q$ for squarefree $q\mid P(z)$.  The sifted set is encoded by
\[
P_\Omega(m)=\prod_{\substack{p\le z\\ m\bmod p\in\Omega_p}}p;
\]
an integer $m\in I$ is counted precisely when $(P_\Omega(m),P(z))=1$.  For squarefree $q\mid P(z)$, the corresponding set $\mathcal A_q$ is exactly the set counted in \eqref{eq:generic-sieve-remainder}.  Applying Lemma \ref{lem:quoted-beta-sieve} with $X=N$, $Q=N^{1/2}$, $\kappa=2$, and the dimension constant from \eqref{eq:generic-dimension-bound} gives
\[
S=N V_\Omega(z)\left(1+O(e^{-cs})\right)
+O\left(\sum_{q\le Q}|r_q|\right)
=N V_\Omega(z)\left(1+O(e^{-cs})\right)+O(N^{1/2+o(1)}),
\]
where $s=\log Q/\log z$.  The constants depend only on the absolute dimension constant in \eqref{eq:generic-dimension-bound}; in particular they are independent of the particular residue classes $\Omega_p$.  The prime $2$ causes no exceptional case here because the assumption $\nu(p)<p$ excludes inadmissibility.  Finally $z\le N^{\beta_N}$ implies $s\ge(2\beta_N)^{-1}$, which proves \eqref{eq:generic-beta-sieve}.
\end{proof}

For $2\le d_1<d_2$ and $h=d_2-d_1$, put
\begin{equation}\label{eq:V12-def}
V(d_1,d_2;h)=
\prod_{\substack{p\le d_1\\p\nmid h}}\left(1-\frac2p\right)
\prod_{\substack{p\le d_1\\p\mid h}}\left(1-\frac1p\right)
\prod_{d_1<p\le d_2}\left(1-\frac1p\right).
\end{equation}
For odd $h$ this product is zero because of the prime $2$, and the corresponding sifted set is empty once $d_1\ge2$.

\begin{lemma}[Two-dimensional fundamental lemma in the long range]\label{lem:two-dim-fundamental}
Let $Y=N^{\beta_N}$, where $\beta_N\to0$ and $\beta_N\log N\to\infty$.  Uniformly for $2\le d_1<d_2\le Y$, with $h=d_2-d_1$, one has
\begin{equation}\label{eq:two-dim-fundamental}
\sum_{N<n\le2N} I_{d_1}(n)I_{d_2}(n)
=N V(d_1,d_2;h)+O\bigl(\eta_N N V(d_1,d_2;h)\bigr)+O(N^{1/2+o(1)}),
\end{equation}
where, for the beta-sieve fundamental lemma, one may take
\begin{equation}\label{eq:eta-beta}
\eta_N\ll \exp(-c/\beta_N)
\end{equation}
with an absolute constant $c>0$.  The $o(1)$ in the final remainder of \eqref{eq:two-dim-fundamental} is absolute and uniform in $d_1,d_2,h$.  Also, uniformly for $2\le d\le Y$,
\begin{align}\label{eq:one-dim-fundamental-for-variance}
\sum_{N<n\le2N}I_d(n)&=N V(d)+O(\eta_N N V(d))+O(N^{1/2+o(1)}),\\
V(d)&=\prod_{p\le d}\left(1-\frac1p\right).\nonumber
\end{align}
\end{lemma}

\begin{proof}
We verify the hypotheses of Lemma \ref{lem:long-interval-beta-sieve}.  Put
$m=n-d_2$ and $h=d_2-d_1$.  The two roughness conditions are
\begin{equation*}
(m,P(d_2))=1,\qquad (m+h,P(d_1))=1,
\end{equation*}
where $P(z)=\prod_{p\le z}p$.  For a prime $p\le d_2$ let $\Omega_p$ be the set of forbidden residue classes for $m\pmod p$.  Thus
\begin{equation*}
\Omega_p=\{0,-h\}\quad(p\le d_1),\qquad
\Omega_p=\{0\}\quad(d_1<p\le d_2),
\end{equation*}
with the two classes in the first set coinciding when $p\mid h$.  If $h$ is odd, then $\Omega_2$ contains both residue classes modulo $2$, and the sifted set is empty; in this case $V(d_1,d_2;h)=0$ and the asserted estimate is immediate.  We may therefore assume that $h$ is even.  Then $\nu(p):=|\Omega_p|<p$ for every $p\le d_2$, and
\begin{equation*}
V(d_1,d_2;h)=\prod_{p\le d_2}\left(1-\frac{\nu(p)}p\right),
\end{equation*}
which is exactly \eqref{eq:V12-def}.

For a squarefree integer $q$ composed of primes at most $d_2$, let $\Omega_q$ be the set of classes modulo $q$ obtained from the Chinese remainder theorem, and write $\nu(q)=|\Omega_q|=\prod_{p\mid q}\nu(p)$.  In the interval
$N-d_2<m\le2N-d_2$, whose length is exactly $N$, we have
\begin{equation*}
\#\{m:\thinspace N-d_2<m\le2N-d_2,\thinspace m\bmod q\in\Omega_q\}
=N\frac{\nu(q)}q+r_q,\qquad |r_q|\le \nu(q)\le2^{\omega(q)}.
\end{equation*}
Thus Lemma \ref{lem:long-interval-beta-sieve}, with $z=d_2$, gives
\begin{equation*}
\sum_{N<n\le2N} I_{d_1}(n)I_{d_2}(n)
=N V(d_1,d_2;h)\left(1+O(e^{-c s_N})\right)+O(N^{1/2+o(1)}),
\end{equation*}
where
\begin{equation*}
s_N=\frac{\log N^{1/2}}{\log d_2}\ge \frac{1}{2\beta_N}.
\end{equation*}
This proves \eqref{eq:two-dim-fundamental} with $\eta_N\ll e^{-c'/\beta_N}$.  The constants and the final $o(1)$ are absolute and uniform in $d_1,d_2,h$.

The one-dimensional estimate \eqref{eq:one-dim-fundamental-for-variance} follows from the same lemma with one forbidden class $0\pmod p$ for each $p\le d$ and sieve level $z=d$; the hypotheses are verified in the same way.
\end{proof}

For real $x\ge2$ we write throughout this section
\[
V(x)=\prod_{p\le x}\left(1-\frac1p\right).
\]
The next lemma is the finite-product cancellation needed to sum the main terms.  The proof uses the exact complete-period mean of the finite local factor.  This is where parity cancellation is used quantitatively; an absolute comparison of $C_y$ with the limiting singular series would lose a factor of order $\log\log Y$.  As a sanity check, $C_y(h)=0$ for odd $h$ because of the prime $2$, while for even $h$ the local factor is often larger than $1$.  Thus the cancellation below is genuinely signed; no absolute-value version of \eqref{eq:finite-series-average} with the same strength can hold.

\begin{lemma}[Exact finite singular-series mean and weighted cancellation]\label{lem:finite-singular-series-average}
For $y\ge2$ define
\begin{equation}\label{eq:Cy-def}
C_y(h)=
\prod_{\substack{p\le y\\ p\nmid h}}\left(1-\frac1{(p-1)^2}\right)
\prod_{\substack{p\le y\\ p\mid h}}\frac{p}{p-1}.
\end{equation}
Then, uniformly for $H,y\ge2$,
\begin{equation}\label{eq:finite-series-average}
\sum_{h\le H}\{C_y(h)-1\}\ll \log(2y).
\end{equation}
Consequently, uniformly for $Y\ge3$,
\begin{equation}\label{eq:weighted-finite-series-cancellation}
\sum_{2\le d_1<d_2\le Y}
\frac{V(d_1)V(d_2)}{d_1d_2}\{C_{d_1}(d_2-d_1)-1\}=O(1).
\end{equation}
\end{lemma}

\begin{proof}
For each prime $p\le y$ put
\[
c_p=1-\frac1{(p-1)^2},\qquad b_p=\frac p{p-1},\qquad \Delta_p=b_p-c_p.
\]
For $p=2$ this gives $c_2=0$ and $\Delta_2=2$; for $p>2$,
\[
\Delta_p=\frac{p}{(p-1)^2}>0.
\]
Thus
\[
C_y(h)=\prod_{p\le y}\bigl(c_p+\Delta_p\one_{p\mid h}\bigr)
      =\sum_{q\mid P(y)} a_y(q)\one_{q\mid h},
\]
where $P(y)=\prod_{p\le y}p$ and all coefficients $a_y(q)$ are non-negative.  Two exact identities are needed.  First,
\[
\sum_{q\mid P(y)} a_y(q)=C_y(0)=\prod_{p\le y}\frac p{p-1}\ll\log(2y)
\]
by Mertens' theorem.  Second, the complete-period average is exactly one:
\[
\sum_{q\mid P(y)}\frac{a_y(q)}q
 =\prod_{p\le y}\left(c_p+\frac{\Delta_p}{p}\right)=1.
\]
Indeed, for $p=2$ this factor is $0+2/2=1$, while for $p>2$ it is
\[
1-\frac1{(p-1)^2}+\frac{1}{p}\frac{p}{(p-1)^2}=1.
\]
Therefore
\[
\sum_{h\le H}C_y(h)
=\sum_{q\mid P(y)}a_y(q)\left\lfloor\frac Hq\right\rfloor
=H\sum_{q\mid P(y)}\frac{a_y(q)}q+O\left(\sum_{q\mid P(y)}a_y(q)\right)
=H+O(\log(2y)).
\]
This proves \eqref{eq:finite-series-average}.  Notice that the floor estimate
$|\lfloor H/q\rfloor-H/q|\le1$ is valid for every $q$, including $q>H$; this is why the truncation of the divisor expansion causes no loss.

For the weighted consequence, fix $y$ and write
\[
A_y(t)=\sum_{h\le t}\{C_y(h)-1\}.
\]
By \eqref{eq:finite-series-average}, $|A_y(t)|\ll\log(2y)$ uniformly in $t$.  Put
\[
w_y(h)=\frac{V(y+h)}{y+h}.
\]  The sequence $w_y(h)$ is non-negative and non-increasing in $h$, since both $V(y+h)$ and $(y+h)^{-1}$ are non-increasing.  Abel summation gives, for $H=Y-y$,
\[
\sum_{1\le h\le H}\{C_y(h)-1\}w_y(h)
=A_y(H)w_y(H)+\sum_{1\le h<H}A_y(h)\{w_y(h)-w_y(h+1)\}.
\]
Thus
\begin{align*}
\left|\sum_{1\le h\le H}\{C_y(h)-1\}w_y(h)\right|
&\le \sup_{t\le H}|A_y(t)|
\left(w_y(H)+\sum_{1\le h<H}(w_y(h)-w_y(h+1))\right)\\
&\le \sup_{t\le H}|A_y(t)|w_y(1).
\end{align*}
Since $w_y(1)\ll1/(y\log(2y))$, we obtain
\[
\left|\sum_{1\le h\le Y-y}\frac{V(y+h)}{y+h}\{C_y(h)-1\}\right|
\ll \frac1y.
\]
Multiplying by $V(y)/y\ll1/(y\log(2y))$ and summing over $2\le y\le Y$ yields
\[
\sum_{2\le y\le Y}\frac{V(y)}y
\left|\sum_{1\le h\le Y-y}\frac{V(y+h)}{y+h}\{C_y(h)-1\}\right|
\ll \sum_{y\ge2}\frac1{y^2\log(2y)}=O(1).
\]
This proves \eqref{eq:weighted-finite-series-cancellation}.
\end{proof}

\begin{lemma}[Summability of main terms and sieve errors]\label{lem:second-moment-summability}
Uniformly for $Y\ge3$,
\begin{align}
\sum_{d\le Y}\frac{V(d)}d&\ll\log\log(2Y),\label{eq:sum-V-over-d}\\
\sum_{2\le d_1<d_2\le Y}\frac{V(d_1,d_2;d_2-d_1)}{d_1d_2}
&\ll (\log\log(2Y))^2,\label{eq:sum-pair-main-positive}\\
\sum_{2\le d_1<d_2\le Y}\frac{V(d_1)V(d_2)}{d_1d_2}
\{C_{d_1}(d_2-d_1)-1\}&=O(1).\label{eq:sum-pair-covariance}
\end{align}
Moreover, if $Y=N^{o(1)}$, then
\begin{align}
N^{1/2+o(1)}\left(\sum_{d\le Y}\frac1d\right)^2&=o(N),\label{eq:additive-pair-summable}\\
N^{1/2+o(1)}\sum_{d\le Y}\frac1d&=o(N).\label{eq:additive-one-summable}
\end{align}
\end{lemma}

\begin{proof}
The first estimate is Mertens' theorem in the form $V(d)\ll1/\log(2d)$.  For the second, use
\[
V(d_1,d_2;d_2-d_1)=V(d_1)V(d_2)C_{d_1}(d_2-d_1).
\]
Thus the left side of \eqref{eq:sum-pair-main-positive} equals
\[
\sum_{d_1<d_2}\frac{V(d_1)V(d_2)}{d_1d_2}
+
\sum_{d_1<d_2}\frac{V(d_1)V(d_2)}{d_1d_2}\{C_{d_1}(d_2-d_1)-1\}.
\]
The first term is $O((\log\log(2Y))^2)$ by \eqref{eq:sum-V-over-d}, and the second is $O(1)$ by Lemma \ref{lem:finite-singular-series-average}.  This proves both \eqref{eq:sum-pair-main-positive} and \eqref{eq:sum-pair-covariance}.  Finally, if $Y=N^{o(1)}$, then $\sum_{d\le Y}1/d\ll\log Y=N^{o(1)}$, so \eqref{eq:additive-pair-summable} and \eqref{eq:additive-one-summable} follow.
\end{proof}

For $Y\ge1$ define the truncated self-rough sum
\[
\LL_Y(n)=\sum_{\substack{d\le Y\\ d<n}}\frac{I_d(n)}d,
\]
and set $V(1)=1$; for real $x\ge2$, $V(x)=\prod_{p\le x}(1-1/p)$.  The next theorem is the point at which Lemma \ref{lem:finite-singular-series-average} is used: that lemma cancels the positive even-gap correlations against the exact negative odd-gap contribution in a signed way.

Before stating the second-moment estimate, we record the parameter ledger for the choice used in the almost-all theorem.  We take
\[
\beta_N=(\log\log N)^{-1/2},\qquad
Y=N^{\beta_N},\qquad Q=N^{1/2},\qquad
s=\frac{\log Q}{\log Y}=\frac1{2\beta_N}.
\]
Then $Y=N^{o(1)}$, $Y\to\infty$, and the beta-sieve relative error satisfies
\[
\eta_N(\log\log N)^2=o(1).
\]
Also
\[
N^{1/2+o(1)}(\log Y)^2
\le N^{1/2+o(1)}(\log N)^2=o(N),
\]
which is the only estimate needed to sum all additive sieve remainders in the proof below.

\begin{theorem}[Truncated second moment]\label{thm:truncated-second-moment}
Let $\beta_N\to0$ with $\beta_N\log N\to\infty$, and put $Y=N^{\beta_N}$.  Let $\eta_N$ be the relative error in Lemma \ref{lem:two-dim-fundamental}.  Then
\begin{equation}\label{eq:truncated-second-moment-precise}
\sum_{N<n\le2N}\left|\sum_{d\le Y}\frac{I_d(n)}d
-
\sum_{d\le Y}\frac{V(d)}d\right|^2
\ll
N+
\eta_N N(\log\log N)^2+o(N).
\end{equation}
In particular, if $\beta_N=(\log\log N)^{-1/2}$, then
\begin{equation}\label{eq:truncated-second-moment}
\sum_{N<n\le2N}\left|\sum_{d\le Y}\frac{I_d(n)}d
-
\sum_{d\le Y}\frac{V(d)}d\right|^2
=O(N)=o\bigl(N(\log\log N)^2\bigr).
\end{equation}
\end{theorem}

\begin{proof}
The term $d=1$ is identically $I_1(n)=1$ on this dyadic interval, and $V(1)=1$ by the empty-product convention.  It is therefore removed before expanding.  Put
\[
\widetilde\LL_Y(n)=\sum_{2\le d\le Y}\frac{I_d(n)}d,
\qquad
\widetilde\mu_Y=\sum_{2\le d\le Y}\frac{V(d)}d.
\]
Then the expression in \eqref{eq:truncated-second-moment-precise} is exactly $\widetilde\LL_Y(n)-\widetilde\mu_Y$, and no $d=1$ cross-term occurs below.  We expand
\[
\sum_{N<n\le2N}|\widetilde\LL_Y(n)-\widetilde\mu_Y|^2
=\sum_{N<n\le2N}\widetilde\LL_Y(n)^2
-2\widetilde\mu_Y\sum_{N<n\le2N}\widetilde\LL_Y(n)
+N\widetilde\mu_Y^2.
\]
The one-dimensional estimate \eqref{eq:one-dim-fundamental-for-variance} gives
\[
\sum_{N<n\le2N}\widetilde\LL_Y(n)
=N\widetilde\mu_Y+O\!\left(\eta_N N\widetilde\mu_Y\right)+O\!\left(N^{1/2+o(1)}\sum_{2\le d\le Y}\frac1d\right).
\]
By Lemma \ref{lem:second-moment-summability}, this error contributes
\[
O\left(\eta_N N(\log\log N)^2\right)+o(N)
\]
after multiplication by $\widetilde\mu_Y\ll\log\log N$.

For the second moment, the diagonal terms contribute
\[
\sum_{2\le d\le Y}\frac1{d^2}\sum_{N<n\le2N} I_d(n)
\ll N\sum_{2\le d\le Y}\frac{V(d)}{d^2}+N^{1/2+o(1)}\sum_{2\le d\le Y}\frac1{d^2}=O(N).
\]
For $2\le d_1<d_2\le Y$, Lemma \ref{lem:two-dim-fundamental} gives
\[
\sum_{N<n\le2N}I_{d_1}(n)I_{d_2}(n)
=N V(d_1,d_2;d_2-d_1)
+O\bigl(\eta_N N V(d_1,d_2;d_2-d_1)\bigr)
+O(N^{1/2+o(1)}).
\]
The main off-diagonal term, after subtracting the corresponding part of $N\widetilde\mu_Y^2$, is
\[
2N\sum_{2\le d_1<d_2\le Y}
\frac{V(d_1)V(d_2)}{d_1d_2}\{C_{d_1}(d_2-d_1)-1\},
\]
which is $O(N)$ by \eqref{eq:sum-pair-covariance}.  The diagonal part of $N\widetilde\mu_Y^2$ contributes only
\[
O\left(N\sum_d \frac{V(d)^2}{d^2}\right)=O(N).
\]
The relative-error terms are bounded by
\[
\eta_N N
\sum_{2\le d_1<d_2\le Y}\frac{V(d_1,d_2;d_2-d_1)}{d_1d_2}
\ll \eta_N N(\log\log N)^2
\]
by \eqref{eq:sum-pair-main-positive}.  The accumulated level-of-distribution remainders are
\[
N^{1/2+o(1)}\left(\sum_{2\le d\le Y}\frac1d\right)^2=o(N)
\]
by \eqref{eq:additive-pair-summable}.  Combining these estimates proves \eqref{eq:truncated-second-moment-precise}.  With
$\beta_N=(\log\log N)^{-1/2}$, Lemma \ref{lem:two-dim-fundamental} gives
$\eta_N\ll\exp(-c\sqrt{\log\log N})$, so the relative-error term is $o(N)$, proving \eqref{eq:truncated-second-moment}.
\end{proof}

We now prove the second main theorem stated in the introduction.

\begin{proof}[Proof of Theorem \ref{thm:main-aa}]
Choose $\beta_N=(\log\log N)^{-1/2}$ and $Y=N^{\beta_N}$.  In this proof
\[
\LL_Y(n)=\sum_{d\le Y}\frac{I_d(n)}d,
\qquad
\mu_Y=\sum_{d\le Y}\frac{V(d)}d,
\]
with the convention $V(1)=1$.  By Mertens' theorem,
\[
\mu_Y
=(e^{-\gamma}+o(1))\sum_{d\le Y}\frac1{d\log d}
=(e^{-\gamma}+o(1))\log\log N.
\]
Theorem \ref{thm:truncated-second-moment} and Chebyshev's inequality give, for each fixed $\eta>0$,
\begin{equation}\label{eq:LY-exception-rate}
\#\{N<n\le2N: |\LL_Y(n)-\mu_Y|>\eta\log\log N\}
\ll_\eta \frac{N}{\log\log N}.
\end{equation}
The tail $d>Y$ has first moment
\begin{align*}
\sum_{N<n\le2N}\sum_{Y<d<n}\frac{I_d(n)}d
&\le \sum_{Y<d<2N}\frac1d\,
\#\{N<n\le2N:\ d<n,\ \Pmin(n-d)>d\}.
\end{align*}
For fixed $d$, put $m=n-d$.  The counted integers $m$ lie in an interval contained in $[1,2N-d]$, and hence their number is at most
\[
\R(2N-d,d)\ll \frac{2N-d}{\log d}+1\ll \frac N{\log d}+1
\]
by Lemma \ref{lem:rough}(ii).  Therefore
\begin{align*}
\sum_{N<n\le2N}\sum_{Y<d<n}\frac{I_d(n)}d
&\ll N\sum_{Y<d<2N}\frac1{d\log d}
+\sum_{Y<d<2N}\frac1d\\
&=O\bigl(N\log(1/\beta_N)\bigr)+O(\log N)
=O(N\log\log\log N).
\end{align*}
Markov's inequality shows that the tail exceeds $\eta\log\log N$ for at most
\begin{equation}\label{eq:tail-exception-rate}
O_\eta\left(\frac{N\log\log\log N}{\log\log N}\right)
\end{equation}
integers $n$ in the dyadic interval.  Thus
\[
\LL(n)=(e^{-\gamma}+o(1))\log\log N
\]
for all but $O_\eta(N\log\log\log N/\log\log N)$ integers in $N<n\le2N$.  Combining this with Theorem \ref{thm:almost-all}, and replacing $\log\log N$ by $\log\log n$ as before, gives \eqref{eq:main-aa-intro}.  Choose $0<\eta_0<(1-e^{-\gamma})/4$.  On the density-one set where \eqref{eq:main-aa-intro} holds with error at most $\eta_0\log\log n$, one has, for all sufficiently large $n$,
\[
\M(n)\le (e^{-\gamma}+\eta_0)\log\log n
< (1-2\eta_0)\log\log n.
\]
Mertens' theorem gives $\sum_{p<n}1/p\ge (1-\eta_0)\log\log n$ for all sufficiently large $n$.  Thus the Erd\H{o}s inequality holds on this density-one set with room to spare; enlarging a constant $C$ absorbs finitely many small values of $n$.

It remains to prove the global quantitative bound \eqref{eq:main-aa-exception-intro}.  Summing the dyadic estimate just proved, the dyadic block $2^j<n\le2^{j+1}$ contributes
\[
O_\eta\!\left(\frac{2^j\log\log\log 2^j}{\log\log 2^j}\right).
\]
The sum over $j\le \log_2 N$ is dominated by the last block:
\[
\sum_{j\le \log_2 N}\frac{2^j\log\log\log 2^j}{\log\log 2^j}
\ll \frac{N\log\log\log N}{\log\log N}.
\]
The initial segment $n\le \sqrt N$ is negligible compared with this bound, and on $\sqrt N<n\le N$ one has $\log\log n\asymp\log\log N$.  This proves \eqref{eq:main-aa-exception-intro}.
\end{proof}

\begin{corollary}[Quantitative exceptional set for the Erd\H{o}s inequality]\label{cor:exception-rate}
Taking, for example, any fixed $0<\eta<(1-e^{-\gamma})/2$, the Erd\H{o}s inequality holds for all sufficiently large $n$ outside a set of size
\[
O_\eta\!\left(\frac{N\log\log\log N}{\log\log N}\right)
\]
in $[1,N]$.
\end{corollary}

\begin{proof}
This is the quantitative estimate \eqref{eq:main-aa-exception-intro} from Theorem \ref{thm:main-aa}, combined with Mertens' theorem and $e^{-\gamma}<1$.
\end{proof}

\begin{remark}[What remains for full variance and limiting laws]\label{rem:full-variance-remains}
Theorem \ref{thm:truncated-second-moment} proves the normalized second-moment estimate needed for the almost-all theorem.  It does not prove the full untruncated bound
\begin{equation}\label{eq:full-L-variance-target}
\sum_{N<n\le2N}\left|\LL(n)-\overline\LL_N\right|^2=O(N),
\qquad
\overline\LL_N=N^{-1}\sum_{N<n\le2N}\LL(n).
\end{equation}
The obstruction is the high-level tail
\[
T_Y(n)=\sum_{Y<d<n}\frac{I_d(n)}d.
\]
Its first moment is small enough for Markov's inequality at the $\log\log n$ scale, but an $L^2$ estimate for this tail is not a routine extension of Lemma \ref{lem:two-dim-fundamental}.  That lemma is a long-interval fundamental-lemma statement with levels $d\le N^{o(1)}$.  Once $d$ is a fixed power of $N$, the Buchstab parameter is finite; once $d>\sqrt{2N}$ on the dyadic interval $N<n\le2N$, a genuinely shifted-prime problem appears.  The next propositions record these two barriers in theorem-level form.
\end{remark}

\subsection{Conditional full-variance and fixed-power directions}
Everything in this subsection is conditional or programmatic and is not used in the proof of Theorems \ref{thm:main}, \ref{thm:almost-all}, or \ref{thm:main-aa}.  The purpose is to isolate exactly which additional estimates would upgrade the normal-order statement to full variance or fixed-power truncation theorems.

\begin{proposition}[The shifted-prime form of the prime tail]\label{prop:shifted-prime-tail}
For $N<n\le2N$ and all sufficiently large $N$,
\begin{equation}\label{eq:prime-tail-identity}
\sum_{\sqrt{2N}<d<n}\frac{I_d(n)}d
=
\frac1{n-1}+
\sum_{\substack{n/2<p<n-\sqrt{2N}\\ p\ \mathrm{prime}}}
\frac1{n-p}.
\end{equation}
More generally, if $Z\ge\sqrt{2N}$, then
\begin{equation}\label{eq:prime-tail-identity-Z}
\sum_{Z<d<n}\frac{I_d(n)}d
=
\frac{\mathbf 1_{\{Z<n-1\}}}{n-1}+
\sum_{\substack{n/2<p<n-Z\\ p\ \mathrm{prime}}}
\frac1{n-p},
\end{equation}
with the prime sum interpreted as empty if $n-Z\le n/2$.
\end{proposition}

\begin{proof}
Put $m=n-d$.  If $d>\sqrt{2N}$ and $I_d(n)=1$, then $\Pmin(m)>d$.  If $m>1$ were composite, then
\[
\Pmin(m)\le \sqrt m\le \sqrt{2N}<d
\]
for large $N$, a contradiction.  Hence $m=1$ or $m=p$ is prime.  The term $m=1$ gives $d=n-1$ and contributes $1/(n-1)$ in \eqref{eq:prime-tail-identity}; in the general form it contributes only when $Z<n-1$, giving the indicator in \eqref{eq:prime-tail-identity-Z}.  If $m=p$ is prime, the condition $\Pmin(p)>d$ is exactly $p>d=n-p$, or $p>n/2$.  The additional condition $d>\sqrt{2N}$ is $p<n-\sqrt{2N}$.  This proves \eqref{eq:prime-tail-identity}, and \eqref{eq:prime-tail-identity-Z} is identical.
\end{proof}

\begin{proposition}[One-dimensional finite-\texorpdfstring{$u$}{u} Buchstab estimate]\label{prop:one-dim-finite-u}
Fix $0<\theta_0<1$.  Uniformly for $3\le d\le N^{\theta_0}$,
\begin{align*}
\sum_{N<n\le2N} I_d(n)
&=\R(2N-d,d)-\R(N-d,d)\\
&=\frac{(2N-d)\omega\!\left(\frac{\log(2N-d)}{\log d}\right)
-(N-d)\omega\!\left(\frac{\log(N-d)}{\log d}\right)}{\log d}
+O_{\theta_0}\!\left(\frac{N}{(\log d)^2}+1\right).
\end{align*}
\end{proposition}

\begin{proof}
The first identity is the substitution $m=n-d$.  Fix a number $d_0\ge3$.  For $d\ge d_0$, the second identity is the uniform Buchstab--de Bruijn estimate for $\R(x,y)$ quoted in Lemma \ref{lem:rough}, applied with $x=2N-d$ and $x=N-d$ and with $y=d$.  Since $d\le N^{\theta_0}$ and $\theta_0<1$, choose $\delta_0>0$ with $1+\delta_0<1/\theta_0$.  Then, for all large $N$ and all $d$ in this range, both $N-d$ and $2N-d$ are at least $d^{1+\delta_0}$.  Thus both applications stay in a finite-$u$ range bounded away from the transition point $u=1$.  For the finitely many $3\le d<d_0$, the displayed error term $O_{\theta_0}(N/(\log d)^2+1)$ is large enough to absorb the trivial bound $O(N)$; the $O(1)$ part only handles endpoint conventions.
\end{proof}

\begin{remark}[What remains in the fixed-power problem]
Proposition \ref{prop:one-dim-finite-u} proves the one-dimensional finite-$u$ input.  It is not enough for Hypothesis \ref{hyp:fixed-power-buchstab}.  The missing part is a two-dimensional Buchstab asymptotic for the pair of forms $m,m+h$ with split sieving levels.  A direct beta-sieve application at $d\asymp N^\theta$ has bounded sieve parameter $s$, so the standard upper and lower sieve functions no longer coalesce.  The long-range proof of Theorem \ref{thm:truncated-second-moment} avoids this because $\theta=\beta_N\to0$ and hence $s\to\infty$.  Thus the fixed-power problem requires a genuine finite-$u$ two-dimensional Buchstab theorem, not only the fundamental lemma used earlier.
\end{remark}

\begin{hypothesis}[Fixed-power two-point Buchstab input]\label{hyp:fixed-power-buchstab}
Fix $0<\theta<1/2$ and put $Y=N^\theta$.  For $2\le d\le Y$ define
\[
M_N(d)=N^{-1}\sum_{N<n\le2N} I_d(n).
\]
Assume that for every $2\le d_1<d_2\le Y$, with $h=d_2-d_1$, there are quantities $\mathcal B_N(d_1,d_2;h)$ and errors $E_N(d_1,d_2;h)$ such that
\begin{equation}\label{eq:fixed-power-two-point-input}
\sum_{N<n\le2N} I_{d_1}(n)I_{d_2}(n)=N\mathcal B_N(d_1,d_2;h)+E_N(d_1,d_2;h).
\end{equation}
Assume moreover the weighted signed covariance and absolute error conditions
\begin{align}
\left|
\sum_{2\le d_1<d_2\le Y}\frac{
\mathcal B_N(d_1,d_2;d_2-d_1)-M_N(d_1)M_N(d_2)}{d_1d_2}
\right|
&\ll_\theta \log\log N, \label{eq:fixed-power-main-summable}\\
\sum_{2\le d_1<d_2\le Y}\frac{|E_N(d_1,d_2;d_2-d_1)|}{d_1d_2}
&\ll_\theta N\log\log N. \label{eq:fixed-power-error-summable}
\end{align}
The first condition is deliberately signed.  An absolute-value version is false in the expected finite-$u$ model: for odd gaps the product $I_{d_1}(n)I_{d_2}(n)$ vanishes identically, giving negative covariance of size comparable to $M_N(d_1)M_N(d_2)$, while even gaps give positive singular-series covariance.  The necessary cancellation is the same even-versus-odd cancellation used in Lemma \ref{lem:finite-singular-series-average}.  These two explicit two-point estimates are the finite-$u$ Buchstab input.  They imply
\begin{equation}\label{eq:fixed-power-variance-hypothesis}
\sum_{N<n\le2N}\left|\sum_{d\le Y}\frac{I_d(n)}d-\sum_{d\le Y}\frac{M_N(d)}d\right|^2=O_\theta(N\log\log N).
\end{equation}
If the right side in \eqref{eq:fixed-power-main-summable} can be improved to $O_\theta(1)$ and the right side in \eqref{eq:fixed-power-error-summable} to $O_\theta(N)$, we call this the sharp fixed-power version; it gives variance $O_\theta(N)$ in \eqref{eq:fixed-power-variance-hypothesis}.
\end{hypothesis}

\begin{proposition}[Consequence of the fixed-power Buchstab input]\label{prop:fixed-power-consequence}
Under Hypothesis \ref{hyp:fixed-power-buchstab}, for all but $O_{\theta,\eta}(N/\log\log N)$ integers $N<n\le2N$, one has
\[
\left|\LL_{\le N^\theta}(n)-\mu_{N,\theta}\right|\le\eta\log\log N,
\qquad
\mu_{N,\theta}:=1+\sum_{2\le d\le N^\theta}\frac{M_N(d)}d.
\]
In particular, the fixed-power input is more than sufficient for the almost-all theorem on that truncated range.  Under the sharp fixed-power version the exceptional-set bound improves to $O_{\theta,\eta}(N/(\log\log N)^2)$.
\end{proposition}

\begin{proof}
The centering above uses exactly the same $d=1$ contribution as $\LL_{\le Y}$, namely $1$, so the $d=1$ term cancels.  Expanding the square over $2\le d\le Y$, using \eqref{eq:fixed-power-two-point-input} for the off-diagonal and the trivial diagonal bound $\sum_d d^{-2}\sum_{N<n\le2N}I_d(n)\ll N$, the signed main-term contribution is exactly the left side of \eqref{eq:fixed-power-main-summable} multiplied by $2N$, and the error contribution is bounded by \eqref{eq:fixed-power-error-summable}.  This gives \eqref{eq:fixed-power-variance-hypothesis}.  Chebyshev with threshold $\eta\log\log N$ gives the displayed exceptional-set bound.  In the sharp version the variance numerator is $O_\theta(N)$, giving $O_{\theta,\eta}(N/(\log\log N)^2)$.
\end{proof}

\begin{hypothesis}[Shifted-prime tail second moment at a specified cutoff]\label{hyp:shifted-prime-tail}
Let $Z_N$ be a specified cutoff satisfying $\sqrt{2N}\le Z_N<N$.  Assume that the centered shifted-prime tail
\[
\mathcal P_{Z_N}(n)=
\sum_{\substack{n/2<p<n-Z_N\\ p\ \mathrm{prime}}}\frac1{n-p}
-
N^{-1}\sum_{N<m\le2N}\sum_{\substack{m/2<p<m-Z_N\\ p\ \mathrm{prime}}}\frac1{m-p}
\]
satisfies
\begin{equation}\label{eq:shifted-prime-L2-hyp}
\sum_{N<n\le2N}|\mathcal P_{Z_N}(n)|^2=O(N).
\end{equation}
A uniform version would require this estimate for every cutoff in a specified family; a dyadic decomposition must require that the corresponding dyadic variances sum to $O(N)$.
\end{hypothesis}

For any dyadic function $F$ on $(N,2N]$, write
\[
\overline F_N=N^{-1}\sum_{N<n\le2N}F(n).
\]
The following criterion is intentionally conditional.  It is not used above.  Its hypotheses must be supplied by an explicit transition-range rough-correlation theorem and by a shifted-prime second-moment theorem; choosing the main terms to be exact counts would make the criterion tautological and is not the intended content.

\begin{proposition}[Conditional full variance criterion]\label{prop:conditional-full-variance}
Fix \(
0<\varepsilon<1/2
\)
and put \(Y_0=N^{1/2-\varepsilon}\).  Suppose the following three estimates hold on \((N,2N]\):
\begin{enumerate}[label=\textup{(\roman*)}]
\item the sharp fixed-power estimate gives
\[
\sum_{N<n\le2N}\left|\LL_{\le Y_0}(n)-\overline{\LL}_{\le Y_0,N}\right|^2=O(N);
\]
\item the transition band
\[
T_{\mathrm{tr}}(n)=\sum_{Y_0<d\le \sqrt{2N}}\frac{I_d(n)}d
\]
satisfies
\[
\sum_{N<n\le2N}\left|T_{\mathrm{tr}}(n)-\overline{T}_{\mathrm{tr},N}\right|^2=O(N);
\]
\item the shifted-prime tail beginning at \(\sqrt{2N}\) satisfies Hypothesis \ref{hyp:shifted-prime-tail} with \(Z_N=\sqrt{2N}\), equivalently
\[
\sum_{N<n\le2N}\left|\sum_{\substack{n/2<p<n-\sqrt{2N}\\ p\ \mathrm{prime}}}\frac1{n-p}-\overline{P}_N\right|^2=O(N).
\]
\end{enumerate}
Then the full variance bound \eqref{eq:full-L-variance-target} holds.
\end{proposition}

\begin{proof}
By Proposition \ref{prop:shifted-prime-tail}, the range \(d>\sqrt{2N}\) differs from the shifted-prime tail in (iii) only by the deterministic term \(1/(n-1)\), whose dyadic variance is \(O(1)\).  Decompose
\[
\LL(n)=\LL_{\le Y_0}(n)+T_{\mathrm{tr}}(n)+T_{\mathrm{prime}}(n).
\]
After subtracting dyadic means, the square of the sum is at most three times the sum of the three squares.  The first two variances are \(O(N)\) by (i) and (ii), and the third is \(O(N)\) by (iii).  This proves \eqref{eq:full-L-variance-target}.
\end{proof}

The three pieces in this criterion should not be conflated.  Hypothesis \ref{hyp:fixed-power-buchstab} concerns the fixed-power rough range \(d\le N^\theta\) with \(\theta<1/2\).  It does not supply the transition-band estimate in (ii), and it is separate from the shifted-prime input in (iii).

\begin{question}[Fixed-power truncations below the shifted-prime range]\label{q:fixed-power-truncations}
Can Hypothesis \ref{hyp:fixed-power-buchstab} be proved unconditionally for every fixed $0<\theta<1/2$?  Equivalently, can one prove a finite-$u$ two-dimensional Buchstab asymptotic for the two affine-linear forms $m,m+h$, uniformly for split levels $d_1<d_2\le N^\theta$, with errors summable against $1/(d_1d_2)$?

This is the natural unconditional intermediate target between Theorem \ref{thm:truncated-second-moment} and the prime tail.  The one-dimensional finite-$u$ analogue is Proposition \ref{prop:one-dim-finite-u}.  It is not obtained by merely replacing $N^{\beta_N}$ by $N^\theta$ in Lemma \ref{lem:two-dim-fundamental}.  When $d\asymp N^\theta$, the Buchstab parameter $u=\log N/\log d$ is finite, so even the one-dimensional main term is governed by $\omega(u)$ rather than only by the local Mertens product $V(d)$.  For the two-point problem one needs the corresponding finite-$u$ two-dimensional Buchstab kernel and the same singular-series cancellation mechanism.
\end{question}

\begin{question}[Truncated Erd\H{o}s--Kac problem]\label{q:truncated-clt}
For $Y=N^{\beta_N}$ in the range of Theorem \ref{thm:truncated-second-moment}, determine the limiting distribution, if any, of the centered truncation
\[
\LL_Y(n)-\sum_{d\le Y}\frac{V(d)}d,
\qquad N<n\le2N.
\]
A three-point singular-series computation, modeled on Remark \ref{rem:parity-singular-series}, is the first nontrivial test.  It is not by itself a central limit theorem: a Gaussian law would require control of all fixed moments, or equivalently of the higher cumulants, together with the correct variance scale.  Thus the three-point calculation is best viewed as the first step toward an Erd\H{o}s--Kac type theorem for the truncation rather than as a consequence of the present second-moment argument.
\end{question}

\section{A pointwise constant and the linear-sieve barrier}\label{sec:pointwise}

We next prove an explicit pointwise bound.  This section replaces the anonymous estimate $\M(n)\ll\log\log n$ by the constant delivered by the standard upper-bound sieve.

\begin{proposition}[Pointwise bound with explicit constant]\label{prop:pointwise-two}
For every $\eps>0$ there is a constant $C_\eps$ such that, for every $n\ge3$,
\begin{equation}\label{eq:pointwise-two}
\M(n)\le (2+\eps)\log\log n+C_\eps.
\end{equation}
Equivalently,
\[
\M(n)\le (2+\eps)\sum_{p<n}\frac1p+O_\eps(1).
\]
\end{proposition}

\begin{proof}
Fix once and for all some $\alpha\in(1/3,1/2)$.  The leading constant obtained below is independent of this choice; $\alpha$ only separates an $O_\alpha(1)$ very-small-prime part from the part controlled by the linear sieve.  Let $A\subset[1,n)$ be pairwise coprime and write $d=n-a$.

First consider selected elements $a>1$ with $\Pmin(a)\le d^\alpha$.  Assign each such $a$ to $p=\Pmin(a)$.  These primes are distinct.  Since $p\le d^\alpha$, we have $d\ge p^{1/\alpha}$ and hence
\[
\sum_{\substack{a\in A\\ \Pmin(a)\le d(a)^\alpha}}\frac1{d(a)}
\le \sum_p p^{-1/\alpha}=O_\alpha(1).
\]
Here convergence only requires $1/\alpha>1$, i.e. $\alpha<1$; the stronger restriction $\alpha<1/2$ is used to keep the linear-sieve parameter $s=1/\alpha+o(1)$ above the edge $2$, while $\alpha>1/3$ keeps $s<3$ so that the formula $F(s)=2e^\gamma/s$ applies.

It remains to count selected elements with $\Pmin(a)>d^\alpha$.  This class is a subset of all $d$ with $\Pmin(n-d)>d^\alpha$.  Fix a small parameter $0<\lambda\le1$ and put $D_j=(1+\lambda)^j$.  The intervals
\[
I_j=(D_j,\min\{(1+\lambda)D_j,n\,\}]
\]
cover $1\le d<n$.  If the last non-empty interval has length $<\lambda D_j/2$, combine it with the preceding interval; this changes only the interval-length constants depending on $\lambda$, and it creates at most $O_\lambda(1)$ additional endpoint contribution.  Thus every interval to which the sieve is applied has length $\asymp_\lambda D_j$ and is contained in $[D_j,2D_j]$.  The parameter $\lambda$ is fixed before Lemma \ref{lem:linear-sieve} is applied, so the threshold in that lemma is allowed to depend on $\lambda$ through these interval-length constants.  On such an interval, $d^\alpha>D_j^\alpha$, so the condition implies that $n-d$ avoids the residue class $n\pmod p$ for every prime $p\le D_j^\alpha$.  Lemma \ref{lem:linear-sieve} gives, uniformly in $n$,
\[
\#\{d\in I_j:\Pmin(n-d)>d^\alpha\}
\le (2+\eps/4)\frac{|I_j|}{\log D_j}
\le (2+\eps/4)\frac{\lambda D_j}{\log D_j}
\]
for all sufficiently large $D_j$, after increasing the implicit threshold in terms of $\lambda$.  Therefore the contribution of this interval is at most
\[
(2+\eps/4)\frac{\lambda}{\log D}.
\]
The finitely many intervals below the threshold contribute $O_{\alpha,\eps,\lambda}(1)$.  If $D=(1+\lambda)^j$, then
\[
\sum_{D<n}\frac{\lambda}{\log D}
=\frac{\lambda}{\log(1+\lambda)}\log\log n+O_\lambda(1).
\]
Choose $\lambda$ so small that
\[
(2+\eps/4)\frac{\lambda}{\log(1+\lambda)}\le 2+\eps.
\]
Combining the small-prime part, the sifted part, and the possible element $a=1$ proves \eqref{eq:pointwise-two}.
\end{proof}

\begin{remark}[Effectivizability of the pointwise constant]\label{rem:effective-pointwise-constant}
The constant $C_\eps$ in Proposition \ref{prop:pointwise-two} is effectivizable in principle once one fixes explicit effective versions of the standard sieve and Mertens inputs.  For example, one may fix $\alpha=2/5$ and choose $0<\lambda<1$ so that
\[
(2+\eps/4)\frac{\lambda}{\log(1+\lambda)}\le 2+\eps.
\]
Let $D_0=D_0(\eps,\lambda)$ be the corresponding threshold from Lemma \ref{lem:linear-sieve}, with $\alpha=2/5$ and error parameter $\eps/4$.  Then the proof gives a valid shape
\[
C_\eps=1+\sum_p p^{-5/2}+\sum_{d\le D_0}\frac1d
 +(2+\eps/4)\frac{\lambda}{\log(1+\lambda)}
 \left(2+\left|\log\log(1+\lambda)\right|+\log\log(D_0+3)\right),
\]
after increasing $D_0$ to absorb the finitely many small intervals.  We do not optimize this expression; its purpose is only to record the dependence once explicit versions of Mertens' theorem and the Rosser--Iwaniec linear sieve are fixed.
\end{remark}

\begin{corollary}[Two-sided order on a positive-density set]\label{cor:positive-density-two-sided}
Fix $\eps>0$ and $0<c<e^{-\gamma}$.  There is a constant $\delta=\delta(c,\eps)>0$ such that, for all sufficiently large $N$, at least $\delta N$ integers $n\le N$ satisfy
\[
 c\log\log n\le \M(n)\le (2+\eps)\log\log n+C_\eps.
\]
In particular, $\M(n)\asymp\log\log n$ on a positive-density set of integers, with explicit constants.
\end{corollary}

\begin{proof}
The upper bound is Proposition \ref{prop:pointwise-two}.  Since $\LL(n)\le \M(n)$ and
\[
\sum_{n\le N}\LL(n)=e^{-\gamma}N\log\log N+O(N)
\]
by Proposition \ref{prop:lower} and Lemma \ref{lem:buchstab-summation}, the lower bound holds on a positive proportion of integers.  Quantitatively, fix $B=2+2\eps$.  Applying Proposition \ref{prop:pointwise-two} with the original $\eps$ and then taking $N$ large enough gives
\[
\M(n)\le B\log\log N
\]
for all $\sqrt N<n\le N$.  If fewer than
\[
\delta N,\qquad \delta=\frac{e^{-\gamma}-c}{2(B-c)},
\]
integers in $[\sqrt N,N]$ had $\LL(n)\ge c\log\log N$, then, using $\LL(n)\le \M(n)$ and the pointwise upper bound, the contribution of $\LL$ over $[\sqrt N,N]$ would be at most
\[
\bigl(c+\delta(B-c)+o(1)\bigr)N\log\log N
<e^{-\gamma}N\log\log N,
\]
contradicting the first moment of $\LL$ after subtracting the negligible initial segment.  Replacing $\log\log N$ by $\log\log n$ on $[\sqrt N,N]$ changes only the constants.
\end{proof}

\begin{remark}[The barrier at $2$]\label{rem:barrier-two}
The constant $2$ in Proposition \ref{prop:pointwise-two} is not an artefact of dyadic summation.  It is the upper linear-sieve factor at the edge $s=2$: in Lemma \ref{lem:linear-sieve}, with $z=D^\alpha$ and $\alpha\uparrow1/2$, the linear-sieve parameter tends to $s=2$ and the main term is $2D/\log D$.  This method therefore cannot reach the conjectural pointwise constant $1$.  The first-hit and divisor-budget certificates in Sections \ref{sec:first-hit} and \ref{sec:replacement} do not by themselves lower this constant: when converted into uniform pointwise estimates they again require the same local window-packing input.  To replace the local factor $2$ by a prime-harmonic factor $1$ one needs information comparable to sharp short-interval prime counting or to the window-packing hypothesis in Section \ref{sec:conditional}.  In particular, a direct bound
\[
|A\cap[n-x,n)|\le \pi(x)+O(x/(\log x)^2)
\]
for all admissible $A,n,x$ would imply the weakened second Hardy-Littlewood estimate discussed in Problem \#855; see Proposition \ref{prop:HL-barrier} below.
\end{remark}

\section{Window reduction and a conditional pointwise theorem}\label{sec:conditional}

This section is conditional and contextual.  It makes explicit the local estimate that would imply Erd\H{o}s's pointwise bound; none of the main unconditional theorems relies on the hypothesis below.  For a pairwise coprime set $A\subset[1,n)$ define
\begin{equation}\label{eq:B-def}
B_{A,n}(x)=|A\cap[n-x,n)|=\#\{a\in A:n-x\le a<n\}.
\end{equation}
Equivalently, $B_{A,n}(x)$ counts selected distances $d=n-a$ with $d\le x$.

\begin{lemma}[Partial summation reduction]\label{lem:partial-summation}
Let $N=n-1$.  For every pairwise coprime $A\subset[1,n)$,
\begin{align}\label{eq:partial-summation-defect}
\sum_{a\in A}\frac1{n-a}-\sum_{p<n}\frac1p
&=\frac{B_{A,n}(N)-\pi(N)}{N}
+\sum_{m=1}^{N-1}\frac{B_{A,n}(m)-\pi(m)}{m(m+1)}.
\end{align}
The boundary term is $O(1)$ uniformly in $A,n$.
\end{lemma}

\begin{proof}
The discrete partial summation identity gives
\[
\sum_{a\in A}\frac1{n-a}=\frac{B_{A,n}(N)}N+
\sum_{m=1}^{N-1}B_{A,n}(m)\left(\frac1m-\frac1{m+1}\right).
\]
The same identity applied to the set of primes $p<n$ gives the corresponding formula with $\pi(m)$ in place of $B_{A,n}(m)$.  Subtracting gives \eqref{eq:partial-summation-defect}.  Finally, every $a>1$ has at least one prime divisor, and pairwise coprimality makes these prime divisors distinct, so $|A|\le\pi(N)+1$; hence the boundary term is $O(1)$.
\end{proof}

\begin{hypothesis}[Conditional Hardy--Littlewood window-packing hypothesis]\label{hyp:window}
There are constants $C>0$ and $\delta>0$ such that, for all $n$, all pairwise coprime $A\subset[1,n)$, and all $3\le x<n$,
\begin{equation}\label{eq:window-hypothesis}
B_{A,n}(x)\le \pi(x)+C\frac{x}{(\log x)^{1+\delta}}.
\end{equation}
\end{hypothesis}

\begin{theorem}[Conditional pointwise Erd\H{o}s bound]\label{thm:conditional}
If Hypothesis \ref{hyp:window} holds, then
\begin{equation}\label{eq:conditional-erdos}
\M(n)\le \sum_{p<n}\frac1p+O(1)
\end{equation}
uniformly in $n$.
\end{theorem}

\begin{proof}
By Lemma \ref{lem:partial-summation}, the defect is at most
\[
O(1)+C\sum_{3\le m<n}\frac{1}{m(\log m)^{1+\delta}}.
\]
The sum converges uniformly in $n$.  Taking the supremum over $A$ proves \eqref{eq:conditional-erdos}.
\end{proof}

The hypothesis can be formulated in an $A$-independent way using the exact local packing problem.  For $x<n$, let $R_n(x)$ be the number of $x$-rough elements in the window,
\[
R_n(x)=\#\{m\in[n-x,n):\Pmin(m)>x\}.
\]
For each $m\in[n-x,n)$ with a prime divisor at most $x$, let
\[
P_x(m)=\{p\le x:p\mid m\}.
\]
Let $\nu_n(x)$ be the matching number of the hypergraph on the vertex set $\{p\le x:p\in\mathbb P\}$ whose nonempty edges are the sets $P_x(m)$.

\begin{lemma}[Exact local packing bound]\label{lem:local-packing}
For every admissible $A\subset[1,n)$,
\begin{equation}\label{eq:local-packing-bound}
B_{A,n}(x)\le R_n(x)+\nu_n(x).
\end{equation}
Moreover $R_n(x)+\nu_n(x)$ is the exact maximum possible size of a pairwise coprime subset of the window $[n-x,n)$.
\end{lemma}

\begin{proof}
If two distinct integers in $[n-x,n)$ share a prime divisor $q$, then $q$ divides their difference, whose absolute value is $<x$; hence $q\le x$.  Therefore the elements with no prime divisor at most $x$ are mutually coprime inside the window and are compatible with every choice of non-rough elements.  Among the remaining elements, pairwise coprimality is exactly the requirement that their sets $P_x(m)$ be pairwise disjoint, which is precisely a hypergraph matching condition.  This proves both the upper bound and exactness.
\end{proof}

Thus it is enough to prove
\[
R_n(x)+\nu_n(x)\le \pi(x)+O\left(\frac{x}{(\log x)^{1+\delta}}\right)
\]
uniformly in $n,x$.  This is a Hardy-Littlewood type window-packing assertion, substantially sharper than what the linear sieve alone gives.

\begin{proposition}[Prime-interval barrier]\label{prop:HL-barrier}
Suppose that, for some constant $C$, the estimate
\begin{equation}\label{eq:too-strong-window}
B_{A,n}(x)\le \pi(x)+C\frac{x}{(\log x)^2}
\end{equation}
holds for all $n$, all pairwise coprime $A\subset[1,n)$, and all $3\le x<n$.  Then, uniformly for $X\ge1$ and $Y\ge3$,
\begin{equation}\label{eq:HL-consequence}
\pi(X+Y)-\pi(X)\le \pi(Y)+O\left(\frac{Y}{(\log Y)^2}\right).
\end{equation}
\end{proposition}

\begin{proof}
Let $n=X+Y+1$ and let
\[
A=\{p\in\mathbb P:X<p<X+Y+1\}.
\]
Then $A$ is pairwise coprime and $B_{A,n}(Y+1)=\pi(X+Y)-\pi(X)+O(1)$, with harmless endpoint conventions.  Applying \eqref{eq:too-strong-window} with $x=Y+1$ gives \eqref{eq:HL-consequence}.
\end{proof}

This is the precise barrier behind Remark \ref{rem:barrier-two}.  Estimate \eqref{eq:HL-consequence} is a conjectural weakened form related to the second Hardy--Littlewood prime-interval inequality.  The classical prime-pair framework and later conditional obstructions are in \cite{HardyLittlewood1923,HensleyRichards1974,Richards1974}, while \cite{MontgomeryVaughan1973} gives the large-sieve upper-bound context; the Erd\H{o}s Problem \#855 page is cited only as a contextual index \cite{Bloom855}.

\section{First-hit certificates}\label{sec:first-hit}

For a prime $p<n$, define
\begin{equation}\label{eq:gp-def}
g_p(n)=\min\{1\le d<n:\Pmin(n-d)=p\}.
\end{equation}
This is finite, since $d=n-p$ is admissible.

\begin{proposition}[Least-prime-factor first-hit certificate]\label{prop:first-hit}
For every $n\ge2$,
\begin{equation}\label{eq:first-hit}
\M(n)\le 1+\sum_{p<n}\frac1{g_p(n)}.
\end{equation}
\end{proposition}

\begin{proof}
Let $A\subset[1,n)$ be pairwise coprime.  Throughout the main part of the proof we restrict to elements $a>1$; the possible element $a=1$ contributes only $1/(n-1)\le1$ and is absorbed at the end.  The element $a=1$, if present, contributes at most $1$.  For every $a>1$ in $A$, assign $p(a)=\Pmin(a)$.  This assignment is injective, because two different elements of $A$ cannot share a prime divisor.  If $a=n-d$, then $\Pmin(n-d)=p(a)$ and hence $g_{p(a)}(n)\le d$.  Thus
\[
\frac1{n-a}=\frac1d\le\frac1{g_{p(a)}(n)}.
\]
Summing over $a>1$ proves \eqref{eq:first-hit}.
\end{proof}

The certificate has a renewal interpretation.  Let
\[
K_p(N)=\#\{m\le N:\Pmin(m)=p\}
=\#\{u\le N/p:\Pmin(u)\ge p\}.
\]

\begin{proposition}[Renewal-cycle bound]\label{prop:renewal}
Let
\[
T_p(N)=\sum_{\substack{n\le N\\p<n}}\frac1{g_p(n)}.
\]
If $K_p(N)\ge1$, then
\begin{equation}\label{eq:renewal-bound}
T_p(N)\le (K_p(N)+1)\left(1+\log\frac{N}{K_p(N)+1}\right)+O(\log N).
\end{equation}
\end{proposition}

\begin{proof}
List the integers $m\le N$ with $\Pmin(m)=p$ as $m_1<\cdots<m_K$.  The condition $p<n\le N$ restricts the $n$-sum, but dropping this restriction can only increase the following upper bound.  Consider first an interval between two consecutive occurrences.  If $m_j<n\le m_{j+1}$, then the closest earlier occurrence is $m_j$ unless $n=m_{j+1}$, in which case distance $0$ is not allowed and the previous occurrence gives distance $m_{j+1}-m_j$.  Thus the contribution from this interval is at most
\[
\sum_{r=1}^{m_{j+1}-m_j}\frac1r=H_{m_{j+1}-m_j}.
\]
The initial range before $m_1$ and the terminal range after $m_K$ contribute at most $O(\log N)$ after enlarging constants.  Hence $T_p(N)$ is bounded by a sum of harmonic numbers $\sum H_{L_j}+O(\log N)$, where the positive lengths $L_j$ have total at most $N$ and the number of complete or terminal pieces is at most $K+1$.  Since $H_L\le1+\log L$ and $x\mapsto 1+\log x$ is concave on $x\ge1$,
\[
\sum_j H_{L_j}\le (K+1)\left(1+\log\frac{N}{K+1}\right)+O(\log N),
\]
after discarding empty pieces.  This proves \eqref{eq:renewal-bound}.
\end{proof}

\section{Divisor-budget replacement}\label{sec:replacement}

For $a>1$ put
\begin{equation}\label{eq:rho-def}
\rho(a)=\sum_{p\mid a}\frac1p.
\end{equation}

\begin{proposition}[Divisor-budget inequality]\label{prop:divisor-budget}
For every pairwise coprime set $A\subset[1,n)$,
\begin{equation}\label{eq:divisor-budget}
\sum_{a\in A}\frac1{n-a}
\le 1+\sum_{p<n}\frac1p
+\sum_{\substack{a\in A\\a>1}}
\left(\frac1{n-a}-\rho(a)\right)_+.
\end{equation}
\end{proposition}

\begin{proof}
For each $a>1$,
\[
\frac1{n-a}\le \rho(a)+\left(\frac1{n-a}-\rho(a)\right)_+.
\]
Pairwise coprimality gives
\[
\sum_{\substack{a\in A\\a>1}}\rho(a)\le\sum_{p<n}\frac1p,
\]
and the possible element $a=1$ contributes at most $1$.
\end{proof}

The following flexible version allows a summable premium and a controlled rough deficit.  Let $d(a)=n-a$ and $u(a)=\log d(a)/\log n$.

\begin{proposition}[Premium-profile replacement]\label{prop:premium}
Let $\kappa(p)\ge0$, and define
\[
B_\kappa(n)=\sum_{p<n}\frac{\kappa(p)}p.
\]
Let $\vartheta:[0,1]\to[0,1]$ be nondecreasing and satisfy
\[
I_\vartheta=\int_0^1\frac{\vartheta(u)}u\,du<\infty.
\]
For $a>1$ set
\[
R_\kappa(a)=\sum_{p\mid a}\frac{1+\kappa(p)}p.
\]
If $G\subseteq A\cap[2,n)$, where $A\subset[1,n)$ is pairwise coprime, and every $a\in G$ satisfies
\begin{equation}\label{eq:premium-condition}
R_\kappa(a)\ge \frac{1-\vartheta(u(a))}{d(a)},
\end{equation}
then
\begin{equation}\label{eq:premium-conclusion}
\begin{split}
\sum_{a\in G}\frac1{d(a)}
&\le \sum_{p<n}\frac1p+B_\kappa(n)+O(1+I_\vartheta).
\end{split}
\end{equation}
\end{proposition}

\begin{proof}
Pairwise coprimality gives
\begin{equation}\label{eq:Rkappa-budget}
\sum_{a\in G}R_\kappa(a)
\le \sum_{p<n}\frac{1+\kappa(p)}p
=\sum_{p<n}\frac1p+B_\kappa(n).
\end{equation}
Let $G_- =\{a\in G:R_\kappa(a)<1/d(a)\}$.  For $a\in G_-$, no prime divisor of $a$ is at most $d(a)$, for otherwise $R_\kappa(a)\ge1/d(a)$.  Hence $\Pmin(a)>d(a)$.  From \eqref{eq:premium-condition},
\[
\frac1{d(a)}-R_\kappa(a)\le\frac{\vartheta(u(a))}{d(a)}\qquad(a\in G_-).
\]
Thus
\[
\sum_{a\in G}\frac1{d(a)}
\le \sum_{a\in G}R_\kappa(a)+
\sum_{\substack{d<n\\\Pmin(n-d)>d}}\frac{\vartheta(\log d/\log n)}d.
\]
The last sum is $O(1+I_\vartheta)$ by Lemma \ref{lem:crude-interval-sieve}: in a dyadic block $D\le d<2D$ there are $O(D/\log D)$ possible $d$, and the block contributes
\[
O\left(\frac{\vartheta(\min(1,\log(2D)/\log n))}{\log D}\right).
\]
Summing over $D=2^j$ and comparing with $\int_0^1\vartheta(u)du/u$ proves the claim, together with \eqref{eq:Rkappa-budget}.
\end{proof}

\begin{corollary}[Fixed power tails are harmless]\label{cor:power-tail}
For every fixed $\theta\in(0,1)$ and every pairwise coprime $A\subset[1,n)$,
\begin{equation}\label{eq:power-tail}
\sum_{\substack{a\in A\\ n-a\ge n^\theta}}\frac1{n-a}
\le \sum_{n^\theta\le p<n}\frac1p+O_\theta(1).
\end{equation}
\end{corollary}

\begin{proof}
Put $Y=n^\theta$ and $d=n-a$.  Split the selected elements with $d\ge Y$ into three classes.  If $a$ has a prime divisor $p$ with $Y\le p\le d$, choose one such $p$; the chosen primes are distinct and pay the class by $\sum_{Y\le p<n}1/p$.  If $a$ has a prime divisor $p\le d$, but all such prime divisors are $<Y$, choose one such $p$; this class contributes at most $\pi(Y)/Y=O_\theta(1)$.  Finally, if $\Pmin(a)>d$, Lemma \ref{lem:crude-interval-sieve} gives total contribution
\[
\ll \sum_{\log_2 Y\le j\le \log_2 n}\frac1j
\ll \log\frac{\log n}{\log Y}+O(1)
=\log(1/\theta)+O(1)=O_\theta(1).
\]
\end{proof}

\section{Forced elements and exchange lemmas}\label{sec:forcing}

The next exchange lemma identifies elements that belong to every extremizer.  For $a=n-d$ and a prime $p\mid a$, define
\begin{equation}\label{eq:lambda-def}
\lambda_p(d)=\min\{e:1\le e<n,\ e\ne d,\ e\equiv d\pmod p\},
\end{equation}
with the convention $1/\lambda_p(d)=0$ if the set is empty.

\begin{proposition}[Nearest-conflict forcing]\label{prop:nearest-conflict}
Let $a=n-d\in[1,n)$.  If
\begin{equation}\label{eq:nearest-condition}
\frac1d>\sum_{p\mid a}\frac1{\lambda_p(d)},
\end{equation}
then $a$ belongs to every pairwise coprime set $A\subset[1,n)$ attaining $\M(n)$.
\end{proposition}

\begin{proof}
Since $[1,n)$ is finite, an extremizer for $\M(n)$ exists.  Let $A$ be an extremizer and suppose $a\notin A$.  Let
\[
B=\{b\in A:(a,b)>1\}.
\]
For each $b\in B$, choose one prime $p_b\mid(a,b)$.  These chosen primes are distinct, because two different elements of $A$ cannot share a prime divisor.  Write $b=n-e_b$.  Since $p_b\mid n-d$ and $p_b\mid n-e_b$, we have $e_b\equiv d\pmod{p_b}$ and $e_b\ne d$.  Hence $e_b\ge\lambda_{p_b}(d)$ and
\[
\sum_{b\in B}\frac1{n-b}
\le\sum_{b\in B}\frac1{\lambda_{p_b}(d)}
\le\sum_{p\mid a}\frac1{\lambda_p(d)}<\frac1d.
\]
Removing $B$ and inserting $a$ preserves pairwise coprimality and strictly increases the total weight, a contradiction.
\end{proof}

\begin{corollary}[Superprofitable elements are forced]\label{cor:superprofitable}
Let $a=n-d>1$.  If
\begin{equation}\label{eq:superprofitable-condition}
\frac1d>\rho(a)=\sum_{p\mid a}\frac1p,
\end{equation}
then $a$ belongs to every extremizer for $\M(n)$.  Moreover, all elements satisfying \eqref{eq:superprofitable-condition} are pairwise coprime.
\end{corollary}

\begin{proof}
Condition \eqref{eq:superprofitable-condition} implies that every prime divisor $p$ of $a$ satisfies $p>d$.  If $e\ne d$ and $e\equiv d\pmod p$, then $e\ge d+p>p$; the alternative $e\le d-p$ is negative.  Thus $1/\lambda_p(d)\le1/p$ with the convention above, and Proposition \ref{prop:nearest-conflict} applies.

If $a_i=n-d_i$ and $a_j=n-d_j$ both satisfy \eqref{eq:superprofitable-condition}, with $d_i<d_j$, then any common prime divisor $q$ would divide $d_j-d_i$, so $q<d_j$.  But all prime divisors of $a_j$ exceed $d_j$, a contradiction.
\end{proof}

\section{CRT sharpness examples}\label{sec:crt}

The following CRT construction is structural rather than quantitative in $n$.  It shows that far-rough residual terms can have large total weight and explains why a direct replacement of each composite by one nearby prime divisor cannot prove \eqref{eq:erdos-question} with bounded loss.  It does not assert that these elements satisfy $1/(n-a)>\rho(a)$; uncontrolled additional large prime factors may still contribute to $\rho(a)$.

\begin{proposition}[Far-rough composite families]\label{prop:far-rough-crt}
For arbitrarily large $M$ there exist an integer $n$ and a pairwise coprime set $A\subset[1,n)$, consisting only of composite integers, such that for every $a\in A$,
\begin{equation}\label{eq:far-rough-condition}
\Pmin(a)>2(n-a),
\end{equation}
and
\begin{equation}\label{eq:far-rough-weight}
\sum_{a\in A}\frac1{n-a}\gg\log\log M.
\end{equation}
\end{proposition}

\begin{proof}
Fix $M$ large.  Choose independently and uniformly a residue class $\rho_q\pmod q$ for every prime $q\le2M$.  Call $3\le k\le M$ admissible if
\[
k\not\equiv\rho_q\pmod q\qquad(q\le2k,\ q\text{ prime}).
\]
The expected weighted count of admissible $k$ is
\[
\sum_{3\le k\le M}\frac1k\prod_{q\le2k}\left(1-\frac1q\right)
\gg\sum_{3\le k\le M}\frac1{k\log k}\gg\log\log M.
\]
Thus there is a deterministic choice of the $\rho_q$ for which the set $K$ of admissible $k$ satisfies
\begin{equation}\label{eq:K-large}
\sum_{k\in K}\frac1k\gg\log\log M.
\end{equation}
For each $k\in K$, choose a distinct prime $Q_k>2M$.  These primes are distinct from all primes $q\le2M$, so all moduli used below are pairwise coprime.  By the Chinese remainder theorem there is an integer $n$ satisfying
\[
n\equiv\rho_q\pmod q\quad(q\le2M),
\qquad
n\equiv k\pmod{Q_k}\quad(k\in K).
\]
Choose $n$ large enough in this residue class so that $n>M$ and $n-k>Q_k$ for every $k\in K$.  Put
\[
A=\{n-k:k\in K\}.
\]
Each $n-k$ is composite, since it is divisible by $Q_k$ and is larger than $Q_k$.  If a prime $q\le2k$ divided $n-k$, then $n\equiv k\pmod q$, hence $\rho_q\equiv k\pmod q$, contradicting admissibility.  Therefore $\Pmin(n-k)>2k$.  Finally, if $k<\ell$ and a prime $q$ divided both $n-k$ and $n-\ell$, then $q\mid\ell-k<\ell$, while $q\mid n-\ell$ and $\Pmin(n-\ell)>2\ell$, impossible.  Hence $A$ is pairwise coprime, and \eqref{eq:far-rough-weight} follows from \eqref{eq:K-large}.
\end{proof}

\begin{remark}[No limsup consequence from this construction]\label{rem:no-limsup}
The integer $n$ produced by Proposition \ref{prop:far-rough-crt} is enormous compared with $M$.  More importantly, the random-residue construction has the Buchstab constant as its natural weighted density:
\[
\mathbb E\sum_{k\in K}\frac1k
=\sum_{3\le k\le M}\frac1k\prod_{q\le 2k}\left(1-\frac1q\right)
=(e^{-\gamma}+o(1))\log\log M.
\]
Thus even a hypothetical polynomial-size realization of this same random-sieve model would only reproduce the constant $e^{-\gamma}$, not a larger limsup constant.  The literal CRT implementation also attaches a distinct prime $Q_k>2M$ to each retained element.  For a random-sieve-size family this forces a modulus at least $\exp(cM)$ for some absolute $c>0$, and in any case the proposition gives no useful relation between $M$ and the final $n$.  It is therefore a structural obstruction to direct divisor replacement, not a quantitative lower bound for
\[
\limsup_{n\to\infty}\frac{\M(n)}{\log\log n}.
\]
Theorem \ref{thm:main} already implies the lower-edge fact that this limsup is at least $e^{-\gamma}$: otherwise a bound $\M(n)\le(e^{-\gamma}-\delta)\log\log n$ for all sufficiently large $n$ would contradict the first moment.
\end{remark}

\begin{question}[Far-rough limsup problem]\label{q:far-rough-limsup}
Does there exist $\delta>0$ and infinitely many $n$ for which one can find a pairwise coprime set $A\subset[1,n)$ satisfying
\[
\Pmin(a)>2(n-a)\qquad(a\in A)
\]
and
\[
\sum_{a\in A}\frac1{n-a}\ge (e^{-\gamma}+\delta)\log\log n?
\]
Such a construction would give a genuine lower bound above the average constant for
\[
\limsup_{n\to\infty}\frac{\M(n)}{\log\log n}.
\]
It cannot be obtained merely by compressing Proposition \ref{prop:far-rough-crt}: one would need both a sub-exponential modulus in the denominator parameter and weighted density above the random-sieve value $e^{-\gamma}$.  Pushing this limsup toward $1$ is the lower-bound counterpart of the upper-edge problem for $\LL(n)$.
\end{question}

\section{Small numerical examples}\label{sec:numerics}

We include a few exact small values only to illustrate the definitions.  They are not used in the proofs.  For each listed $n$, the maximum was obtained by exhaustive search over pairwise-coprime subsets of $[1,n)$.  Equivalently, one may form the conflict graph on $[1,n)$, joining two integers with gcd greater than $1$, and solve the resulting maximum-weight independent-set problem with weights $1/(n-a)$; the small cases in the table were checked by direct backtracking.  The table gives one maximizing set $A$, the value of $\M(n)$, and the self-rough lower weight $\LL(n)$.
\[
\begin{array}{c|c|c|c}
 n & A & \M(n) & \LL(n)\\ \hline
 16 & \{15,14,13,11,1\} & 2.1000 & 1.6000\\
 20 & \{19,18,17,13,11,7,5,1\} & 2.2835 & 1.6399\\
 30 & \{29,28,27,25,23,19,17,13,11,1\} & 2.4900 & 1.3452\\
 40 & \{39,38,37,35,31,29,23,17,11,1\} & 2.3978 & 1.6198
\end{array}
\]
The examples show the typical local shape near the diagonal: several integers immediately below $n$ can be selected because their prime divisors are forced to be disjoint by the short distance between them.  For instance, at $n=20$ the element $18=20-2$ is compatible with the selected odd numbers, while $19=20-1$ and $17=20-3$ are primes.  The forced-core criterion of Corollary \ref{cor:superprofitable} detects many such near-diagonal choices when the reciprocal weight $1/(n-a)$ exceeds the available reciprocal prime-divisor budget.

\section{Historical and bibliographic remarks}\label{sec:historical-notes}

The inequality \eqref{eq:erdos-question} is the corrected form of a question of Erd\H{o}s; the original historical sources are \cite{Erdos1977,Erdos1980}.  Bloom's index lists the modern formulation as Erd\H{o}s Problem~\#1210 \cite{Bloom1210}, and it also points to related shifted-prime entries including Problems~\#460 and~\#950 \cite{Bloom460,Bloom950}.  The index is cited here only to identify the modern problem number and nearby entries; the original Erd\H{o}s papers are the primary historical references.

The earlier formulation mentioned in the Bloom entry concerned shifted reciprocals of primes $q_i\in(n,m]$ and is connected with the related entries just cited.  This connection also explains why the prime-only special case appears naturally in the present problem: taking $A$ to consist of primes already forces one to confront sums such as
\[
        \sum_{p<n}\frac1{n-p}.
\]
Consequently the pointwise problem contains short-interval prime phenomena that are not expected to be accessible by the averaging method used for Theorem \ref{thm:main}.

Pairwise-coprime subsets of intervals and extremal questions for coprime sets have a substantial surrounding literature.  Erd\H{o}s and S\'ark\"ozy studied large sets of pairwise coprime integers in intervals \cite{ErdosSarkozy1993}, while Ahlswede and Khachatrian studied maximal sets of integers not containing many pairwise coprime elements \cite{AhlswedeKhachatrian1995}.  The present problem is different: the elements are measured by the shifted weights $1/(n-a)$, and the objective is a weighted extremal problem near the endpoint $n$ rather than a cardinality extremum.  The high-tail prime-only obstruction is connected with shifted-prime and prime-interval questions, including the Hardy--Littlewood prime-pair framework \cite{HardyLittlewood1923}, the Hensley--Richards obstruction to the second Hardy--Littlewood inequality \cite{HensleyRichards1974,Richards1974}, and the Montgomery--Vaughan large-sieve upper bound for primes in intervals \cite{MontgomeryVaughan1973}.

\section{Concluding problems}

Theorem \ref{thm:main} determines the first moment of the exact extremal quantity in Erd\H{o}s's problem.  The proof avoids short-interval prime estimates: after averaging over $n$, both the lower and upper bounds reduce to global counts of rough numbers.  This is why the leading constant is the limiting Buchstab constant $e^{-\gamma}$.

The pointwise problem remains subtler, but Theorem \ref{thm:main-aa} proves that it holds for almost all $n$.  The proof has two parts: first Theorem \ref{thm:almost-all} reduces $\M(n)$ to the self-rough lower model $\LL(n)$ outside a controlled exceptional set, and then the two-dimensional sieve calculation in Section \ref{sec:distribution} proves concentration of the truncated self-rough sum at $e^{-\gamma}\log\log n$.  The singular-series computation in Remark \ref{rem:covariance} is made quantitative in the normalized truncated second-moment estimate \eqref{eq:truncated-second-moment}.  A full untruncated $O(N)$ variance and any fixed $O(1)$-scale limiting law would require an $L^2$ treatment of the high-level tail.  Section \ref{sec:distribution} now records the exact shifted-prime form of the range $d>\sqrt{2N}$ and a conditional full-variance criterion, while Questions \ref{q:fixed-power-truncations} and \ref{q:truncated-clt} isolate the fixed-power Buchstab and truncated Erd\H{o}s--Kac problems.

A second high-ceiling direction is the far-rough limsup problem in Question \ref{q:far-rough-limsup}.  The CRT examples show that one-prime divisor replacement cannot be sharp for far-rough residual terms, but at present they do not give an effective limsup lower bound: their natural weighted density is only the Buchstab value $e^{-\gamma}$, and the modulus is far too large.  Any construction exceeding that constant would probe the worst-case sharpness of the benchmark constant.

A secondary term
\[
\sum_{n\le N}\M(n)=e^{-\gamma}N\log\log N+c_2N+o(N)
\]
would require substantially finer information, including the average gap between $\M(n)$ and the self-rough lower model.  We do not pursue that here.

\appendix
\section{The beta-sieve input used in Section \ref{sec:distribution}}\label{app:beta-sieve}

This appendix records the imported sieve theorem used in Section \ref{sec:distribution} and derives the particular long-interval residue-class estimate needed there.  We do not reprove the construction of the Rosser--Iwaniec beta-sieve weights; the result below is the standard fundamental lemma of the beta sieve, stated in the notation of this paper and specialized to squarefree-supported local densities.  It is the large-$s$ form of Iwaniec--Kowalski \cite[Theorem~6.9]{IwaniecKowalski}, equivalently Friedlander--Iwaniec \cite[Ch.~6, Fundamental Lemma]{FriedlanderIwaniec}.  In those sources the distribution level is often denoted by $D$; here it is denoted by $Q$.  The sifting level is $z$, and the sieve parameter is
\[
        s=\frac{\log Q}{\log z}.
\]
Some sources state the error as a sieve-function term $E_\kappa(s)$ tending to zero.  The only quantitative strength used in this paper is $E_2(s)(\log\log N)^2=o(1)$ for $s=(2\beta_N)^{-1}$; the displayed exponential form is the standard large-$s$ consequence of the beta-sieve fundamental lemma.

\begin{theorem}[Imported beta-sieve fundamental lemma]\label{thm:imported-beta-sieve}
Fix $\kappa\ge1$ and $C_\kappa\ge1$.  There are constants
\[
        s_0=s_0(\kappa,C_\kappa),\qquad c_\kappa>0,
\]
with the following property.  Let $z\ge2$, let $Q=z^s$ with $s\ge s_0$, and let $g$ be a multiplicative function on squarefree divisors of
\[
        P(z)=\prod_{p\le z}p
\]
such that $0\le g(p)<1$ for every $p\le z$ and
\begin{equation}\label{eq:appendix-dimension}
\prod_{w\le p<z}(1-g(p))^{-1}
\le C_\kappa\left(\frac{\log z}{\log w}\right)^\kappa
        \qquad(2\le w\le z).
\end{equation}
Put
\[
        V_g(z)=\prod_{p\le z}(1-g(p)).
\]
Then there are beta-sieve weights $\lambda_q^+$ and $\lambda_q^-$, depending only on $g,z,Q$, such that
\begin{enumerate}[label=\textup{(\roman*)}]
\item $\lambda_q^\pm=0$ unless $q$ is squarefree, $q\mid P(z)$, and $q\le Q$;
\item $\lambda_1^\pm=1$ and $|\lambda_q^\pm|\le1$ for all $q$;
\item for every integer $a$,
\begin{equation}\label{eq:appendix-sieve-inequality}
        \sum_{q\mid(a,P(z))}\lambda_q^-
        \le \mathbf 1_{(a,P(z))=1}
        \le \sum_{q\mid(a,P(z))}\lambda_q^+;
\end{equation}
\item the weighted local sums satisfy
\begin{equation}\label{eq:appendix-weighted-main}
        \sum_{q\mid P(z)}\lambda_q^\pm g(q)
        =V_g(z)\left(1+O_{\kappa,C_\kappa}(e^{-c_\kappa s})\right).
\end{equation}
\end{enumerate}
The constants in \eqref{eq:appendix-weighted-main} depend only on $\kappa$ and $C_\kappa$, not on the particular values of $g(p)$ satisfying \eqref{eq:appendix-dimension}.
\end{theorem}

The following is the sequence form used in Section \ref{sec:distribution}.  It is included to make the passage from Theorem \ref{thm:imported-beta-sieve} to the two-form sieve completely explicit.

\begin{corollary}[Sequence form of the imported beta sieve]\label{cor:imported-beta-sequence}
Let the hypotheses of Theorem \ref{thm:imported-beta-sieve} hold.  Let $\mathcal A$ be a finite sequence of integers and suppose that, for every squarefree $q\mid P(z)$ with $q\le Q$, one has
\[
        |\mathcal A_q|=Xg(q)+r_q,
        \qquad
        \mathcal A_q=\{a\in\mathcal A:q\mid a\}.
\]
Then
\begin{equation}\label{eq:appendix-sequence-conclusion}
S(\mathcal A,z):=\#\{a\in\mathcal A:(a,P(z))=1\}
=XV_g(z)\left(1+O_{\kappa,C_\kappa}(e^{-c_\kappa s})\right)
+O\left(\sum_{\substack{q\le Q\\ q\mid P(z)}}|r_q|\right).
\end{equation}
\end{corollary}

\begin{proof}
By \eqref{eq:appendix-sieve-inequality},
\[
\sum_{q\mid P(z)}\lambda_q^-|\mathcal A_q|
\le S(\mathcal A,z)\le
\sum_{q\mid P(z)}\lambda_q^+|\mathcal A_q|.
\]
Substituting $|\mathcal A_q|=Xg(q)+r_q$, using $|\lambda_q^\pm|\le1$, and applying \eqref{eq:appendix-weighted-main} gives both upper and lower estimates with the same main term and with error bounded by the remainder sum in \eqref{eq:appendix-sequence-conclusion}.
\end{proof}

\begin{proposition}[Long-interval forbidden residue classes]\label{prop:appendix-long-interval-residue}
Let $I$ be an interval of length $N$, let $2\le z\le N^{\beta_N}$ with $\beta_N\to0$, and put $Q=N^{1/2}$.  For each prime $p\le z$ let $\Omega_p\subset\mathbb Z/p\mathbb Z$ satisfy $|\Omega_p|\le2$ and $|\Omega_p|<p$.  Then, uniformly in the sets $\Omega_p$,
\[
\#\{m\in I:m\bmod p\notin\Omega_p\text{ for all }p\le z\}
=
N\prod_{p\le z}\left(1-\frac{|\Omega_p|}{p}\right)
\left(1+O(e^{-c/\beta_N})\right)+O(N^{1/2+o(1)}),
\]
with an absolute constant $c>0$.  The constants are independent of the particular residue classes.  If $z$ is bounded, the same conclusion follows by the Chinese remainder theorem, with the $O(1)$ periodic error absorbed by $N^{1/2+o(1)}$.
\end{proposition}

\begin{proof}
For bounded $z$ the assertion is immediate from periodicity modulo $P(z)$.  Assume now that $z\to\infty$.  We verify the residue-class model used in Lemma \ref{lem:long-interval-beta-sieve}.  Let $I$ be an interval of length $N$.  For each prime $p\le z$ let $\Omega_p\subset\mathbb Z/p\mathbb Z$ be a set of forbidden residue classes, put $\nu(p)=|\Omega_p|$, and assume $\nu(p)\le2$ and $\nu(p)<p$.  For squarefree $q\mid P(z)$, the Chinese remainder theorem gives a set $\Omega_q\subset\mathbb Z/q\mathbb Z$ of
\[
        \nu(q)=\prod_{p\mid q}\nu(p)
\]
classes.  Hence
\begin{equation}\label{eq:appendix-interval-distribution}
\#\{m\in I:m\bmod q\in\Omega_q\}
=N\frac{\nu(q)}q+r_q,
        \qquad |r_q|\le\nu(q)\le2^{\omega(q)}.
\end{equation}
Thus $g(q)=\nu(q)/q$ is multiplicative on squarefree $q$.  The dimension condition with $\kappa=2$ follows because
\[
\log\prod_{w\le p<z}\left(1-\frac{\nu(p)}p\right)^{-1}
\le \sum_{w\le p<z}\frac2p+O\left(\sum_p\frac1{p^2}\right)
\ll \log\frac{\log z}{\log w}+O(1),
\]
where the prime $2$ contributes at most an absolute factor; this is exactly \eqref{eq:appendix-dimension} for a fixed absolute $C_2$.
Finally,
\begin{equation}\label{eq:appendix-tau2-sum}
\sum_{q\le N^{1/2}}2^{\omega(q)}
\le \sum_{q\le N^{1/2}}\tau(q)
\ll N^{1/2}\log N
=N^{1/2+o(1)}.
\end{equation}
Taking $Q=N^{1/2}$ and $z\le N^{\beta_N}$ gives
\[
        s=\frac{\log Q}{\log z}\ge \frac1{2\beta_N}.
\]
For $\beta_N=(\log\log N)^{-1/2}$, the imported theorem gives
\[
        e^{-c_2s}(\log\log N)^2=o(1),
\]
which is the only quantitative strength needed in Theorem \ref{thm:truncated-second-moment}.  The inadmissible case $\nu(2)=2$, which occurs in the two-form application precisely for odd gaps $h$, is not passed to the beta sieve; the sifted set is then empty and is treated separately before Lemma \ref{lem:long-interval-beta-sieve} is invoked.

Combining \eqref{eq:appendix-interval-distribution}, \eqref{eq:appendix-dimension}, \eqref{eq:appendix-tau2-sum}, and Corollary \ref{cor:imported-beta-sequence} proves Proposition \ref{prop:appendix-long-interval-residue}; this is exactly the input used in Lemma \ref{lem:long-interval-beta-sieve}.
\end{proof}
\section*{Acknowledgements}
The author acknowledges the use of OpenAI's ChatGPT during the preparation of this manuscript. While it was used for ideation, formulation, proof exploration and refinement, narrowing the search space, programming, LaTeX formatting and other forms of orchestration, the author nonetheless takes full responsibility for the accuracy of the final contents of this paper.

The accompanying Lean formalisation was developed by the author with OpenAI Codex.  The author likewise takes full responsibility for the formal statements' fidelity to the results of this paper and for the released source.

\end{document}